\documentclass[a4paper,10pt,reqno]{amsart}

\usepackage{amsmath}
\usepackage{amsthm}
\usepackage{amsfonts}
\usepackage{amssymb}
\usepackage{mathrsfs}
\usepackage[shortlabels]{enumitem}
\usepackage{graphicx}
\usepackage{subfigure}
\usepackage{color}
\usepackage[dvipsnames]{xcolor}

\newtheorem{theorem}{Theorem}[section]
\numberwithin{theorem}{section}
\newtheorem{prop}[theorem]{Proposition}
\newtheorem{lemma}[theorem]{Lemma}

\theoremstyle{definition}
\newtheorem{definition}[theorem]{Definition}
\newtheorem{rem}[theorem]{Remark}

\newtheorem{example}[theorem]{Example}

\numberwithin{equation}{section}

\newcommand{\N}{\mathbb{N}}
\newcommand{\Z}{\mathbb{Z}}

\newcommand{\R}{\mathbb{R}}

\newcommand{\C}{\mathbb{C}}
\newcommand{\K}{\mathbb{K}}

\newcommand\e{\mathrm{e}}
\newcommand\I{\mathrm{i}}
\newcommand\re{\operatorname{Re}}

\newcommand{\rd}{\mathrm{d}}

\newcommand\dom{\mathcal D}
\newcommand\ran{\mathcal R}

\newcommand\lbar\overline

\newcommand\eps\varepsilon
\renewcommand\epsilon\varepsilon
\renewcommand\rho\varrho
\newcommand\al\alpha
\newcommand\lm\lambda

\newcommand\ds\displaystyle

\newcommand\p\partial

\newcommand{\To}{\Longrightarrow}

\newcommand{\tolong}{\longrightarrow}

\newcommand{\gnr}{\stackrel{gnr}{\rightarrow}}
\newcommand{\gsr}{\stackrel{gsr}{\rightarrow}}
\newcommand{\s}{\stackrel{s}{\rightarrow}}

\newcommand{\slong}{\stackrel{s}{\longrightarrow}}

\newcommand{\supp}{\operatorname{supp}}

\newcommand{\beq}{\begin{equation}}
\newcommand{\eeq}{\end{equation}}
\newcommand{\bmat}{\begin{pmatrix}}
\newcommand{\emat}{\end{pmatrix}}

\author{Sabine B\"ogli}
\address[S.\ B\"ogli]{
Department of Mathematics,
Imperial College London,
Huxley Building,
180 Queen's Gate, London SW7 2AZ, UK}
\email{s.boegli@imperial.ac.uk}

\thanks{
}

\usepackage[update,prepend]{epstopdf}

\title[Local convergence of spectra and pseudospectra]{Local convergence of spectra and pseudospectra}

\begin{document}



\date{\today}

\begin{abstract}
We prove local convergence results for the spectra and pseudospectra of sequences of  linear operators acting in different Hilbert spaces and converging in generalised strong resolvent sense to an operator with possibly non-empty essential spectrum.
We establish local spectral exactness outside the limiting essential spectrum, local $\varepsilon$-pseudospectral exactness outside the limiting essential $\varepsilon$-near spectrum, and discuss properties of these two notions including perturbation results.
\end{abstract}

\maketitle

\section{Introduction}
We address the problem of convergence of spectra and pseudospectra for a sequence $(T_n)_{n\in\N}$ of closed linear operators approximating an operator $T$. We establish regions $K\subset\C$ of local convergence,
\begin{align}
\lim_{n\to\infty}\sigma(T_n)\cap K&=\sigma(T)\cap K, \label{eq.conv.of.spec}\\
 \lim_{n\to\infty}\overline{\sigma_{\eps}(T_n)}\cap K&=\overline{\sigma_{\eps}(T)}\cap K, \quad \eps>0, \label{eq.conv.of.pseudospec}
\end{align}
where the limits are defined appropriately. Recall that for $\eps>0$ the $\eps$-pseudospectrum is defined as the open set
$$\sigma_{\eps}(T):=\Big\{\lm\in\C:\,\|(T-\lm)^{-1}\|>\frac{1}{\eps}\Big\},$$
employing the convention that $\|(T-\lm)^{-1}\|=\infty$ for $\lm\in\sigma(T)$ (see~\cite{Trefethen-2005} for an overview). 

We allow the operators $T$, $T_n$, $n\in\N$, to act in different Hilbert spaces $H$, $H_n$, $n\in\N$, and require only convergence in so-called generalised strong resolvent sense, i.e.\ the sequence of projected resolvents $((T_n-\lm)^{-1}P_{H_n})_{n\in\N}$ shall converge strongly to $(T-\lm)^{-1}P_H$ in a common larger Hilbert space.

The novelty of this paper lies in its general framework which is applicable to a wide range of operators $T$ and approximating sequences $(T_n)_{n\in\N}$:
1) We do not assume selfadjointness as in~\cite[Section~VIII.7]{reedsimon1},~\cite[Section~9.3]{weid1}, or boundedness of the operators as in~\cite{stummel2,vainikko,fillmore,ChandlerWilde-Lindner} (see also~\cite{chatelin} for an overview).
2) The operators may have non-empty essential spectrum, in contrast to the global spectral exactness results for operators with compact resolvents~\cite{Boegli-chap1,Osborn-1975-29,wolff}.
3) The results are applicable, but not restricted to, the domain truncation method for differential operators~\cite{Marletta-2010,Brown-Marletta-2001,Brown-Marletta-2003,Brown-2004-24,Reddy-1993-5,Davies-2000-43} 
and to the Galerkin (finite section) method~\cite{levitin,boulton,strauss,Lewin-Sere,Boettcher-1994,Boettcher-Wolf-1997,Boettcher-Silbermann}.
4) Our assumptions are weaker than the convergence in operator norm~\cite{Harrabi-1998} or in (generalised) norm resolvent sense~\cite{Boegli-2014-80,Hansen-2011-24}.

Regarding convergence of spectra (see~\eqref{eq.conv.of.spec}), the aim is to establish \emph{local spectral exactness} of the approximation $(T_n)_{n\in\N}$ of $T$, i.e.\
\begin{enumerate}
\item[\rm(1)] \emph{local spectral inclusion}: For every $\lm\in\sigma(T)\cap K$ there exists a sequence $(\lm_n)_{n\in\N}$ of $\lm_n\in\sigma(T_n)\cap K$, $n\in\N$, with $\lm_n\to\lm$ as $n\to\infty$;
\item[\rm(2)] \emph{no spectral pollution}: If there exists a sequence $(\lm_n)_{n\in I}$ of $\lm_n\in\sigma(T_n)\cap K$, $n\in I$, with $\lm_n\to \lm$ as $n\in I$, $n\to\infty$, then $\lm\in\sigma(T)\cap K$.
\end{enumerate}
Concerning pseudospectra (see~\eqref{eq.conv.of.pseudospec}), we define \emph{local $\eps$-pseudospectral exactness, 
-inclusion, -pollution} in an analogous way by replacing all spectra by the closures of $\eps$-pseudospectra.

In general, spectral exactness is a major challenge for non-selfadjoint problems.
In the selfadjoint case, it is well known that generalised strong resolvent convergence
implies spectral inclusion, and if the resolvents converge even in norm, then spectral
exactness prevails~\cite[Section~9.3]{weid1}. 
In the non-selfadjoint case, norm
resolvent convergence excludes spectral pollution; however,
the approximation need not be spectrally inclusive~\cite[Section~IV.3]{kato}. 
Stability problems are simpler when passing from spectra to pseudospectra; in particular, they converge ($\eps$-pseudospectral exactness) under generalised norm resolvent convergence~\cite[Theorem~2.1]{Boegli-2014-80}. However, if the resolvents converge only strongly, $\eps$-pseudospectral pollution may occur.

In the two main results (Theorems~\ref{mainthmspectralexactness},~\ref{thmexactness2}) we prove local spectral exactness outside the \emph{limiting essential spectrum} $\sigma_{\rm ess}\big((T_n)_{n\in\N}\big)$, and local $\eps$-pseudospectral exactness outside the \emph{limiting essential $\eps$-near spectrum} $\Lambda_{{\rm ess},\eps}\big((T_n)_{n\in\N}\big)$.
The notion of limiting essential spectrum was introduced by Boulton, Boussa\"id and Lewin in~\cite{boulton} for Galerkin approximations of selfadjoint operators. Here we generalise it to our more general framework, 
$$\sigma_{\rm ess}\big((T_n)_{n\in\N}\big):=\left\{\lm\in\C:\,
\begin{array}{l} \exists\,I\subset\N\,\exists\,x_n\in\dom(T_n),\,n\in I,\,\text{with}\\
\|x_n\|=1,\,x_n\stackrel{w}{\to}0,\,\|(T_n-\lm)x_n\|\to 0
\end{array}\right\},$$
and further to pseudospectral theory,
$$\Lambda_{{\rm ess},\eps}\big((T_n)_{n\in\N}\big):=\left\{\lm\in\C:\,
\begin{array}{l} \exists\,I\subset\N\,\exists\,x_n\in\dom(T_n),\,n\in I,\,\text{with}\\
\|x_n\|=1,\,x_n\stackrel{w}{\to}0,\,\|(T_n-\lm)x_n\|\to \eps
\end{array}\right\}.$$
Outside these problematic parts, we prove convergence of the ($\eps$-pseudo-) spectra with respect to the Hausdorff metric.
In the case of pseudospectra, the problematic part is the whole complex plane if $T$ has constant resolvent norm on an open set (see~Theorem~\ref{thmconstresnorm} and also~\cite{Boegli-2014-80}).

The paper is organised as follows. In Section~\ref{sectionspectra} we study convergence of spectra.
First we prove local spectral exactness outside the limiting essential spectrum (Theorem~\ref{mainthmspectralexactness}). Then we  establish properties of $\sigma_{\rm ess}\big((T_n)_{n\in\N}$, including a spectral mapping theorem (Theorem~\ref{mappingthm}) which implies a perturbation result for $\sigma_{\rm ess}\big((T_n)_{n\in\N}\big)$ (Theorem~\ref{thmreldisccomppert}).
In Section~\ref{sectionpseudospectra} we address pseudospectra and prove local $\eps$-pseudospectral convergence  (Theorem~\ref{thmexactness2}). 
Then we establish properties of the limiting essential $\eps$-near spectrum $\Lambda_{{\rm ess}, \eps}\big((T_n)_{n\in\N}\big)$ including a perturbation result (Theorem~\ref{prop.pert.pseudo}).
In the final Section~\ref{sectionapplications}, applications to the Galerkin method of block-diagonally dominant matrices and to the domain truncation method of perturbed constant-coefficient PDEs are studied.

Throughout this paper we denote by $H_0$ a separable infinite-dimensional Hilbert space. 
The notations $\|\cdot\|$ and $\langle\cdot,\cdot\rangle$ refer to the norm and scalar product of $H_0$.
Strong and weak convergence of elements in $H_0$ is denoted by $x_n\to x$ and $x_n\stackrel{w}{\to}0$, respectively.
The space $L(H)$ denotes the space of all bounded operators acting in a Hilbert space $H$.
Norm and strong operator convergence in $L(H)$ is denoted by $B_n\to B$ and $B_n\s B$, respectively.
An identity operator is denoted by $I$; scalar multiples $\lm I$ are written as $\lm$.
Let $H, H_n\subset H_0$, $n\in\N,$ be closed subspaces and $P=P_H:H_0\to H$, $P_n=P_{H_n}:H_0\to H_n, \,n\in\N,$ be the orthogonal projections onto the respective subspaces  and suppose that they converge strongly, $P_n\s P$.
Throughout, let $T$ and $T_n$, $n\in\N$, be closed, densely defined linear operators acting in the spaces $H$, $H_n$, $n\in\N$, respectively. 
The domain, spectrum, point spectrum, approximate point spectrum and resolvent set 
of $T$ are denoted by $\dom(T)$, $\sigma(T)$, $\sigma_p(T)$, $\sigma_{\rm app}(T)$ and $\rho(T)$, 
respectively, and the Hilbert space adjoint operator of $T$ is~$T^*$.
For non-selfadjoint operators there exist (at least) five different definitions for the essential spectrum which all coincide in the selfadjoint case; for a discussion see~\cite[Chapter~IX]{edmundsevans}.
Here we use
$$\sigma_{\rm ess}(T):=\left\{\lm\in\C:\,\exists\, (x_n)_{n\in\N}\subset\dom(T) \text{ with }\|x_n\|=1,\,x_n\stackrel{w}{\to}0,\,\|(T-\lm)x_n\|\to 0\right\},$$
which corresponds to $k=2$ in~\cite{edmundsevans}. 
The remaining spectrum $\sigma_{\rm dis}(T):=\sigma(T)\backslash\sigma_{\rm ess}(T)$ is called the discrete spectrum.
For a subset $\Omega\subset\C$ we denote $\Omega^*:=\{\overline{z}:\,z\in\Omega\}$.
Finally, for two compact subsets $\Omega,\Sigma\subset\C$, their Hausdorff distance is ${\rm d_H}(\Omega,\Sigma):=\max\big\{\sup_{z\in\Omega}{\rm dist}(z,\Sigma),\sup_{z\in\Sigma}{\rm dist}(z,\Omega)\big\}$
where ${\rm dist}(z,\Sigma):=\inf_{w\in \Sigma}|z-w|$.

\section{Local convergence of spectra}\label{sectionspectra}
In this section we address the problem of local spectral exactness.
In Subsection~\ref{subsectionmainresults} we introduce the limiting essential spectrum $\sigma_{\rm ess}\big((T_n)_{n\in\N}$ and state the main result (Theorem~\ref{mainthmspectralexactness}). Then we establish properties of $\sigma_{\rm ess}\big((T_n)_{n\in\N}$ in Subsection~\ref{subsectionpropertiessigmaess}, including a spectral mapping theorem (Theorem~\ref{mappingthm}) which implies a perturbation result for $\sigma_{\rm ess}\big((T_n)_{n\in\N}\big)$ (Theorem~\ref{thmreldisccomppert}). At the end of the section, in Subsection~\ref{subsectionproofofmainresults}, we prove the main result and illustrate it for the example of Galerkin approximations of perturbed Toeplitz operators.

\subsection{Main convergence result}\label{subsectionmainresults}

The following definition of {generalised} strong and norm resolvent convergence is due to Weidmann~\cite[Section~9.3]{weid1} who considers selfadjoint operators.

\begin{definition}\label{defresconv}
\begin{enumerate}
\item[\rm i)] The sequence $(T_n)_{n\in\N}$ is said to \emph{converge in generalised strong resolvent sense} to $T$, 
denoted by $T_n\gsr T$, if there exist $n_0\in\N$ and $\lm_0\in\underset{n\geq n_0}{\bigcap}\rho(T_n)\cap\rho(T)$ with\vspace{-2mm}
$$(T_n-\lm_0)^{-1}P_n\slong (T-\lm_0)^{-1}P, \quad n\to\infty.$$
\item[\rm ii)] The sequence $(T_n)_{n\in\N}$ is said to \emph{converge in generalised norm resolvent sense} to $T$, 
denoted by $T_n\gnr T$, if there exist $n_0\in\N$ and $\lm_0\in\underset{n\geq n_0}{\bigcap}\rho(T_n)\cap\rho(T)$ with\vspace{-2mm}
$$(T_n-\lm_0)^{-1}P_n\tolong (T-\lm_0)^{-1}P, \quad n\to\infty.$$
\end{enumerate}
\end{definition}

The following definition generalises a notion introduced in~\cite{boulton} for the Galerkin method of selfadjoint operators.
\begin{definition}\label{defsigmae}
The \emph{limiting essential spectrum} of $(T_n)_{n\in\N}$ is defined as  
\begin{align*}
\sigma_{\rm ess}\big((T_n)_{n\in\N}\big):=\left\{\lm\in\C:\,
\begin{array}{l} \exists\,I\subset\N\,\exists\,x_n\in\dom(T_n),\,n\in I,\,\text{with}\\
\|x_n\|=1,\,x_n\stackrel{w}{\to}0,\,\|(T_n-\lm)x_n\|\to 0
\end{array}\right\}.
\end{align*}
\end{definition}

The following theorem is the main result of this section. We characterise regions where approximating sequences $(T_n)_{n\in\N}$ are locally spectrally exact and establish spectral convergence with respect to the Hausdorff metric.
\begin{theorem}\label{mainthmspectralexactness}
 \begin{enumerate}[label=\rm{\roman{*})}] 
\item Assume that $T_n\gsr T$ and $T_n^*\gsr T^*$. 
Then spectral pollution is confined to the set \beq \sigma_{\rm ess}\left((T_n)_{n\in\N}\right)\cup\sigma_{\rm ess}\left((T_n^*)_{n\in\N}\right)^*,\label{eqbadset}\eeq
and for every isolated $\lm\in\sigma(T)$ that does not belong to the set in~\eqref{eqbadset},   there exists a sequence of $\lm_n\in\sigma(T_n), \,n\in\N,$ with $\lm_n\to\lm, \,n\to\infty$.

\item Assume that $T_n\gsr T$ and $T_n,\,n\in\N,$ all have compact resolvents. 
Then claim{\rm~i)} holds with~\eqref{eqbadset} replaced by the possibly smaller set 
\beq\label{eqbadset2}
\sigma_{\rm ess}\left((T_n^*)_{n\in\N}\right)^*.
\eeq

\item Suppose that the assumptions of {\rm i)} or {\rm ii)} hold, and let $K\subset\C$ be a compact subset such that $K\cap\sigma(T)$ is discrete and belongs to the interior of $K$.
If the intersection of $K$ with the set in~\eqref{eqbadset} or~\eqref{eqbadset2}, respectively, is contained in $\sigma(T)$,
then
$${\rm d_H}\big(\sigma(T_n)\cap K,\sigma(T)\cap K\big)\tolong 0, \quad n\to\infty.$$
\end{enumerate}
\end{theorem}


\subsection{Properties of the limiting essential spectrum}\label{subsectionpropertiessigmaess}

In this subsection we establish properties that the limiting essential spectrum shares with the essential spectrum (see \cite[Sections~IX.1,2]{edmundsevans}).

The following result follows from a standard diagonal sequence argument; we omit the proof.

\begin{prop}\label{propsigmaeclosed}
The limiting essential spectrum $\sigma_{\rm ess}\left((T_n)_{n\in\N}\right)$ is a closed subset of $\C$.
\end{prop}

The limiting essential spectrum satisfies a mapping theorem.

\begin{theorem}\label{mappingthm}
 Let $\lm_0\in\underset{n\in\N}{\bigcap}\rho(T_n)$ and $\lm\neq \lm_0$. Then the following are equivalent:
\begin{enumerate}
 \item[\rm(1)] $\lm\in\sigma_{\rm ess}\left((T_n)_{n\in\N}\right)$;
 \item[\rm(2)] $(\lm-\lm_0)^{-1}\in\sigma_{\rm ess}\big(\big((T_n-\lm_0)^{-1}\big)_{n\in\N}\big)$.
\end{enumerate}
\end{theorem}

\begin{proof}
``(1)$\,\To\,$(2)'':
Let $\lm\in\sigma_{\rm ess}\left((T_n)_{n\in\N}\right)$. 
By Definition~\ref{defsigmae} of the limiting essential spectrum, there exist $I\subset\N$ and $x_n\in\dom(T_n), \,n\in I,$ such that $\|x_n\|=1$,  $x_n\stackrel{w}{\to} 0$ and $\|(T_n-\lm)x_n\|\to 0$.
Note that $\|(T_n-\lm_0)x_n\|\to |\lm-\lm_0|\neq 0$, hence there exists $N\in\N$ such that $\|(T_n-\lm_0)x_n\|>0$ for every $n\in I$ with $n\geq N$.
Define $$y_n:=\frac{(T_n-\lm_0)x_n}{\|(T_n-\lm_0)x_n\|}, \quad n\in I, \quad n\geq N.$$
Then $\|y_n\|=1$ and $$y_n=\frac{(T_n-\lm)x_n}{\|(T_n-\lm_0)x_n\|}+\frac{(\lm-\lm_0)x_n}{\|(T_n-\lm_0)x_n\|}\stackrel{w}{\tolong}0, \quad n\to\infty.$$
Moreover, we calculate
\begin{align*}
 \left\|\left((T_n-\lm_0)^{-1}-(\lm-\lm_0)^{-1}\right)y_n\right\|&=\left\|\frac{x_n-(\lm-\lm_0)^{-1}(T_n-\lm_0)x_n}{\|(T_n-\lm_0)x_n\|}\right\|\\
&=|\lm-\lm_0|^{-1}\,\frac{\|(T_n-\lm)x_n\|}{\|(T_n-\lm_0)x_n\|}\tolong 0, \quad n\to\infty.
\end{align*}
This implies $(\lm-\lm_0)^{-1}\in\sigma_{\rm ess}\big(\big((T_n-\lm_0)^{-1}\big)_{n\in\N}\big)$.

``(2)$\,\To\,$(1)'':
It is easy to check that if there exist an infinite subset $I\subset\N$ and $y_n\in H_n$, $n\in I$, with $\|y_n\|=1$, $y_n\stackrel{w}{\to}0$ and $ \left\|\left((T_n-\lm_0)^{-1}-(\lm-\lm_0)^{-1}\right)y_n\right\|\to 0$,
then $$x_n:=\frac{(T_n-\lm_0)^{-1}y_n}{\|(T_n-\lm_0)^{-1}y_n\|}, \quad n\in I, $$
satisfy  $\|x_n\|=1$, $x_n\stackrel{w}{\to} 0$ and $\|(T_n-\lm)x_n\|\to 0$.
\end{proof}

\begin{rem}
For the Galerkin method, Theorem~\ref{mappingthm} is different from the spectral mapping theorem ~\cite[Theorem 7]{boulton} for semi-bounded selfadjoint operators.
Whereas in Theorem~\ref{mappingthm} the resolvent of the approximation, i.e.\ $(P_nT|_{\ran(P_n)}-\lm_0)^{-1}$, is considered, 
the result in~\cite{boulton} is formulated in terms of the approximation of the resolvent, i.e.\  $P_n(T-\lm_0)^{-1}|_{\ran(P_n)}$, which is in general not easy to compute.
\end{rem}

The essential spectrum is contained in its limiting counterpart.

\begin{prop}\label{propsigmaess}
\begin{enumerate}
\item[\rm i)] Assume that $T_n\gsr T$. Then
$\sigma_{\rm ess}(T)\subset\sigma_{\rm ess}\big((T_n)_{n\in\N}\big).$
\item[\rm ii)] If $T_n\gnr T$, then $\sigma_{\rm ess}(T)=\sigma_{\rm ess}\left((T_n)_{n\in\N}\right)$.
\end{enumerate}
\end{prop}

For the proof we use the following elementary result.

\begin{lemma}\label{AndAJndJnew}
Assume that $T_n\gsr T$. Then
for all $x\in \dom(T)$ there exists a sequence of elements $x_n\in\dom(T_n), \,n\in\N,$ with $\|x_n\|=1$, $n\in\N$, and
    \beq \|x_n-x\|+\|T_nx_n-Tx\|\tolong 0, \quad n\to\infty.\label{eq.disc.conv}\eeq
\end{lemma}

\begin{proof}
By Definition~\ref{defresconv}~i) of $T_n\gsr T$, there exist $n_0\in\N$ and $\lm_0\in\rho(T)$ such that $\lm_0\in\rho(T_n)$, $n\geq n_0$, and
\beq (T_n-\lm_0)^{-1}P_n\slong (T-\lm_0)^{-1}P, \quad n\to\infty.\label{resconv}\eeq
Let $x\in\dom(T)$ and define $$y_n:=(T_n-\lm_0)^{-1}P_n(T-\lm_0)x\in\dom(T_n), \quad n\geq n_0.$$
Then, using $P_n\s P$ and \eqref{resconv}, it is easy to verify that $\|y_n-x\|\to 0$ and $\|T_ny_n-Tx\|\to 0$.
In particular, there exists $n_1\geq n_0$ such that $y_n\neq 0$ for all $n\geq n_1$.
Now~\eqref{eq.disc.conv} follows for arbitrary normalised $x_n\in\dom(T_n)$, $n<n_1$, and $x_n:=y_n/\|y_n\|$, $n\geq n_1$.
\end{proof}

\begin{proof}[Proof of Proposition{\rm~\ref{propsigmaess}}]
i) Let $\lm\in\sigma_{\rm ess}(T)$. By definition, there exist an infinite subset $I\subset\N$ and $x_k\in\dom(T),\,k\in I,$ with $\|x_k\|=1$, $x_k\stackrel{w}{\to}0$ and 
\beq\|(T-\lm)x_k\|\tolong 0, \quad k\to\infty.\label{eqlminsigmae}\eeq
Let $k\in I$ be fixed.
Since $T_n\gsr T$, Lemma~\ref{AndAJndJnew} implies that there exists a sequence of elements $x_{k;n}\in\dom(T_n),\,n\in\N,$ such that $\|x_{k;n}\|=1$,  $\|x_{k;n}-x_k\|\to 0$ and $\|T_nx_{k;n}-Tx_k\|\to 0$ as $n\to\infty$.
Let $(n_k)_{k\in I}$ be a sequence  such that $n_{k+1}>n_k, \,k\in I,$ and 
\begin{align}\label{ynkapproximatedxk}
\|x_{k;n_k}-x_k\|<\frac{1}{k}, \quad \|T_{n_k}x_{k;n_k}-Tx_k\|<\frac{1}{k}, \quad k\in I.
\end{align}
Define $\widetilde x_{k}:=x_{k;n_k}\in\dom(T_{n_k}), \quad k\in I.$
Then \eqref{ynkapproximatedxk} and $x_k\stackrel{w}{\to}0$ imply 
$\widetilde x_{k}\stackrel{w}{\to}0$ as $k\in I$, $k\to\infty$.
Moreover,~\eqref{eqlminsigmae} and~\eqref{ynkapproximatedxk} yield
$\|(T_{n_k}-\lm)\widetilde x_{k}\|\to 0$ as $k\in I$, $k\to\infty$.
Altogether we have $\lm\in \sigma_{\rm ess}\left((T_n)_{n\in\N}\right)$.

ii) 
 The inclusion $\sigma_{\rm ess}(T)\subset\sigma_{\rm ess}\left((T_n)_{n\in\N}\right)$ follows from i).
Let $\lm\in\sigma_{\rm ess}\left((T_n)_{n\in\N}\right)$.
By the assumption $T_n\gnr T$, there exist $n_0\in\N$ and $\lm_0\in\rho(T)$ with $\lm_0\in\rho(T_n)$, $n\geq n_0$, and
$(T_n-\lm_0)^{-1}P_n\to (T-\lm_0)^{-1}P$.
The mapping result established in Theorem~\ref{mappingthm} implies 
$(\lm-\lm_0)^{-1}\in\sigma_{\rm ess}\big(\left((T_n-\lm_0)^{-1}\right)_{n\geq n_0}\big)$.
So there are an infinite subset $I\subset\N$ and $x_n\in H_n$, $n\in I$, with $\|x_n\|=1$, $x_n\stackrel{w}{\to} 0$
and $$\left\|\left((T_n-\lm_0)^{-1}-(\lm-\lm_0)^{-1}\right)x_n\right\|\tolong 0, \quad n\in I,\quad  n\to\infty.$$
Moreover, in the limit $n\to\infty$ we have
\begin{align*}
\left\|\left((T-\lm_0)^{-1}P-(\lm-\lm_0)^{-1}\right)x_n\right\|
&\leq 
\left\|\left((T_n-\lm_0)^{-1}P_n-(\lm-\lm_0)^{-1}\right)x_n\right\|\\
&\quad +\left\|(T_n-\lm_0)^{-1}P_n-(T-\lm_0)^{-1}P\right\|
\tolong 0.
\end{align*}
Hence $$0\neq (\lm-\lm_0)^{-1}\in\sigma_{\rm ess}\big((T-\lm_0)^{-1}P\big)\subset \sigma_{\rm ess}\big((T-\lm_0)^{-1}\big)\cup\{0\}.$$
Now $\lm\in\sigma_{\rm ess}(T)$ follows from the mapping theorem~\cite[Theorem~IX.2.3, $k$=2]{edmundsevans} for the essential spectrum.
\end{proof}

Now we study sequences of operators and perturbations that are compact or relatively compact in a sense that is appropriate for sequences.
We use Stummel's notion of
discrete compactness of a sequence of bounded operators (see~\cite[Definition~3.1.(k)]{stummel1}).

\begin{definition}\label{defdiscreteconvofop}
Let $B_n\in L(H_n)$, $n\in\N$. 
The sequence  $(B_n)_{n\in\N}$ is said to be \emph{discretely compact}\index{compactness!discretely compact operator sequence} if for each infinite subset $I\subset\N$ and each bounded sequence of elements $x_n\in H_n, \,n\in I,$ there 
exist $x\in H$ and an infinite subset $\widetilde I\subset I$ so that $\|x_n-x\|\to 0$ as $n\in \widetilde I$, $n\to\infty$. 
\end{definition}

\begin{prop}\label{propdisccomp}
\begin{enumerate}
\item[\rm i)]
If $T_n\in L(H_n)$, $n\in\N$, are so that $(T_n)_{n\in\N}$ is a discretely compact sequence and $(T_n^*P_n)_{n\in\N}$ is strongly convergent, then $$\sigma_{\rm ess}\big((T_n)_{n\in\N}\big)=\{0\}.$$
If, in addition, $(T_nP_n)_{n\in\N}$ is strongly convergent, then $$\sigma_{\rm ess}\big((T_n^*)_{n\in\N}\big)^*=\{0\}.$$

\item[\rm ii)]
If there exists $\lm_0\in\underset{n\in\N}{\bigcap}\rho(T_n)$ such that $((T_n-\lm_0)^{-1})_{n\in\N}$ is a discretely compact sequence and $\big((T_n^*-\overline{\lm_0})^{-1}P_n\big)_{n\in\N}$ is strongly convergent, then $$\sigma_{\rm ess}\big((T_n)_{n\in\N}\big)=\emptyset.$$
If, in addition, $\big((T_n-\lm_0)^{-1}P_n\big)_{n\in\N}$ is strongly convergent, then $$\sigma_{\rm ess}\big((T_n^*)_{n\in\N}\big)^*=\emptyset.$$
\end{enumerate}
\end{prop}

For the proof we need the following lemma. Claim~ii)  is the ``discrete'' analogue for operator sequences of the property of operators to be completely continuous.

\begin{lemma}\label{lemmadisccompletelycont}
 Let $B_n\in L(H_n)$, $n\in\N$, and $B_0\in L(H_0)$ with $B_n^*P_n\s B_0^*$. 
Consider an infinite subset $I\subset\N$ and elements $x\in H_0$ and $x_n\in H_n$, $n\in I$, such that $x_n\stackrel{w}{\to}x$ as $n\in I$, $n\to\infty$.
\begin{enumerate}
\item[\rm i)]
We have $x\in H$ and $B_nx_n\stackrel{w}{\to}B_0x\in H$ as $n\in I$, $n\to\infty$.
\item[\rm ii)]
If $(B_n)_{n\in\N}$ is discretely compact, then $B_nx_n\to B_0x$ as $n\in I$, $n\to\infty$.
\end{enumerate}
\end{lemma}

\begin{proof}
i) 
First note that, for any $z\in H_0$, we have $$\langle x_n,z\rangle=\langle x_n,P_nz\rangle\tolong \langle x,Pz\rangle=\langle Px,z\rangle, \quad n\in I,\quad n\to\infty,$$
 and hence $x_n\stackrel{w}{\to} Px$. By the uniqueness of the weak limit, we obtain  $x=Px\in H$.
The weak convergence $B_nx_n\stackrel{w}{\to}B_0x$ is shown analogously, and also $B_nx_n=P_nB_nx_n\stackrel{w}{\to}PB_0x$ which proves $B_0x=PB_0x\in H$.

ii)
Assume that there exist an infinite subset $I_0\subset I$ and $\eps>0$ such that \beq \|B_nx_n-B_0x\|>\eps, \quad n\in I_0. \label{eqcontradictionBnxntoAx}\eeq
 Since the sequence $(x_n)_{n\in I_0}$ is  bounded and $(B_n)_{n\in\N}$ is a discretely compact sequence, by Definition~\ref{defdiscreteconvofop} there exists an infinite subset $\widetilde I\subset I_0$ such that $(B_nx_n)_{n\in\widetilde I}\subset H_0$
is convergent (in $H_0$) to some $y\in H$. Then, by claim~i), the strong convergence $B_n^*P_n\s B_0^*$ and the weak convergence $x_n\stackrel{w}{\to} x$ imply $B_nx_n\stackrel{w}{\to}B_0x\in H$ as $n\in \widetilde I$, $n\to\infty$.
By the uniqueness of the weak limit, we obtain $y=B_0x$, and therefore $(B_nx_n)_{n\in\widetilde I}$ converges to $B_0x$.
The obtained contradiction to~\eqref{eqcontradictionBnxntoAx} proves the claim.
\end{proof}

\begin{proof}[Proof of Proposition{\rm~\ref{propdisccomp}}]
i)
Let $I\subset\N$ be an infinite subset, and let $x_n\in\dom(T_n)$, $n\in I$, satisfy $\|x_n\|=1$ and $x_n\stackrel{w}{\to}0$.
Lemma~\ref{lemmadisccompletelycont}~ii) implies $T_nx_n\to 0$. Now the first claim follows immediately.

Now assume that, in addition,  $(T_nP_n)_{n\in\N}$ is strongly convergent. Then, by~\cite[Proposition~2.10]{Boegli-chap1}, the sequence $(T_n^*)_{n\in\N}$ is discretely compact. Now the second claim follows analogously as the first claim.

ii) 
The assertion follows from i) and the mapping result in Theorem~\ref{mappingthm}; note that $(\lm-\lm_0)^{-1}\neq 0$ for all $\lm\in\C$.
\end{proof}

The limiting essential spectrum is invariant under (relatively) discretely compact perturbations.

\begin{theorem}\label{thmreldisccomppert}
\begin{enumerate}
\item[\rm i)]
Let $B_n\in L(H_n),\,n\in\N$.
 If the sequence $(B_n)_{n\in\N}$ is discretely compact and $(B_n^*P_n)_{n\in\N}$ is strongly convergent, then 
$$\sigma_{\rm ess}\left((T_n+B_n)_{n\in\N}\right)=\sigma_{\rm ess}\left((T_n)_{n\in\N}\right).$$
If, in addition, $(B_nP_n)_{n\in\N}$ is strongly convergent, then 
$$\sigma_{\rm ess}\left(((T_n+B_n)^*)_{n\in\N}\right)^*=\sigma_{\rm ess}\left((T_n^*)_{n\in\N}\right)^*.$$

\item[\rm ii)]
For $n\in\N$ let $A_n$ be a closed, densely defined operator in $H_n$.
If there exists $\lm_0\in\underset{n\in\N}{\bigcap}\rho(T_n)\cap\underset{n\in\N}{\bigcap}\rho(A_n)$
such that the sequence \vspace{-1mm}
$$\left((T_n-\lm_0)^{-1}-(A_n-\lm_0)^{-1}\right)_{n\in\N}$$ is discretely compact and
$\big(\left((T_n-\lm_0)^{-1}-(A_n-\lm_0)^{-1}\right)^*P_n\big)_{n\in\N}$ is strongly convergent, 
then \vspace{-0.5mm}
$$\sigma_{\rm ess}\big((T_n)_{n\in\N}\big)=\sigma_{\rm ess}\big((A_n)_{n\in\N}\big).$$
If, in addition,
$\big(\left((T_n-\lm_0)^{-1}-(A_n-\lm_0)^{-1}\right)P_n\big)_{n\in\N}$ is strongly convergent,
then  \vspace{-0.5mm}
 $$\sigma_{\rm ess}\big((T_n^*)_{n\in\N}\big)^*=\sigma_{\rm ess}\big((A_n^*)_{n\in\N}\big)^*.$$
\end{enumerate}
\end{theorem}

\begin{proof}
i)
Let $I\subset\N$ be an infinite subset, and let $x_n\in\dom(T_n)$, $n\in I$, satisfy $\|x_n\|=1$ and $x_n\stackrel{w}{\to}0$.
Lemma~\ref{lemmadisccompletelycont}~ii) implies $B_nx_n\to 0$. Now the first claim follows immediately.

Now assume that, in addition,  $(B_nP_n)_{n\in\N}$ is strongly convergent. Then, by~\cite[Proposition~2.10]{Boegli-chap1}, the sequence $(B_n^*)_{n\in\N}$ is discretely compact. Now the second claim follows analogously as the first claim.

ii)
 By~i), we have \vspace{-1mm}
$$\sigma_{\rm ess}\big(\left((T_n-\lm_0)^{-1}\right)_{n\in\N}\big)=\sigma_{\rm ess}\big(\left((A_n-\lm_0)^{-1}\right)_{n\in\N}\big).$$
Now the first claim follows from the mapping result in Theorem~\ref{mappingthm}.

If, in addition,
$\big(\left((T_n-\lm)^{-1}-(A_n-\lm)^{-1}\right)P_n\big)_{n\in\N}$ is strongly convergent, then \cite[Proposition~2.10]{Boegli-chap1}
implies that $\big(\left((T_n-\lm_0)^{-1}-(A_n-\lm_0)^{-1}\right)^*P_n\big)_{n\in\N}$ is discretely compact. Now the second claim follows analogously.
\end{proof}

\subsection{Proof of local spectral convergence result and example} \label{subsectionproofofmainresults}

In this subsection we prove the local spectral exactness result in Theorem~\ref{mainthmspectralexactness} and then illustrate it for the Galerkin method of perturbed Toeplitz operators.

First we establish relations of the limiting essential spectrum with the following two notions of limiting approximate point spectrum and region of boundedness (introduced by  Kato~\cite[Section~VIII.1]{kato}).

\begin{definition}\label{defsigmaappandDelbtab}
The \emph{limiting approximate point spectrum} of $(T_n)_{n\in\N}$ is defined~as 
$$\sigma_{\rm app}\big((T_n)_{n\in\N}\big):=\left\{\lm\in\C:\,\exists\,I\subset\N\,\exists\,x_n\in\dom(T_n), \,n\in I, \text{with}
\begin{array}{l}\|x_n\|=1,\\ \|(T_n-\lm)x_n\|\to 0\end{array}\right\},$$
and the \emph{region of boundedness} of $(T_n)_{n\in\N}$ is
\begin{align*}&\Delta_b\left((T_n)_{n\in\N}\right):=\left\{\lm\in\C:\,\exists\,n_0\in\N\,\text{with}\begin{array}{l}\lm\in\rho(T_n),\,n\geq n_0,\\ \left(\|(T_n-\lm)^{-1}\|\right)_{n\geq n_0} \text{bounded}\end{array}\right\}.\end{align*}
\end{definition}

The following lemma follows easily from Definitions~\ref{defsigmae} and~\ref{defsigmaappandDelbtab}.

\begin{lemma} \label{lemmasigmaappofTn1}
\begin{enumerate}[label=\rm{\roman{*})}] 
\item We have $\sigma_{\rm ess}\big((T_n)_{n\in\N}\big)\subset \sigma_{\rm app}\big((T_n)_{n\in\N}\big).$

\item In general, $$\C\backslash\Delta_b\big((T_n)_{n\in\N}\big)=\sigma_{\rm app}\big((T_n)_{n\in\N}\big)\cup \sigma_{\rm app}\big((T_n^*)_{n\in\N}\big)^*.$$
 If $T_n,\,n\in\N,$ all have compact resolvents, then
\begin{align*}
\C\backslash\Delta_b\left((T_n)_{n\in\N}\right)
&=\sigma_{\rm app}\big((T_n)_{n\in\N}\big)= \sigma_{\rm app}\big((T_n^*)_{n\in\N}\big)^*.
\end{align*}
\end{enumerate}
\end{lemma}

Under generalised strong resolvent convergence we obtain the following relations.

\begin{prop} \label{lemmasigmaappofTn}
\begin{enumerate}[label=\rm{\roman{*})}] 
\item If $T_n\gsr T$, then $$
\sigma_{\rm app}(T)\subset\sigma_{\rm app}\left((T_n)_{n\in\N}\right).$$

\item If $T_n\gsr T$, then $$\sigma_{\rm ess}\big((T_n^*)_{n\in\N}\big)^*\subset\sigma_{\rm app}\big((T_n^*)_{n\in\N}\big)^*\subset\sigma_{\rm ess}\big((T_n^*)_{n\in\N}\big)^*\cup\sigma_p(T^*)^*.$$

\item If $T_n\gsr T$ and $T_n^*\gsr T^*$, then 
\begin{align*}
\sigma_{\rm app}\big((T_n)_{n\in\N}\big)&=\sigma_{\rm ess}\big((T_n)_{n\in\N}\big)\cup\sigma_p(T),\\
\sigma_{\rm app}\big((T_n^*)_{n\in\N}\big)^*&=\sigma_{\rm ess}\big((T_n^*)_{n\in\N}\big)^*\cup\sigma_p(T^*)^*.
\end{align*}
\end{enumerate}
\end{prop}

\begin{proof}
i)
 The proof  is analogous to the proof of Proposition~\ref{propsigmaess}; the only difference is that here weak convergence of the considered elements is not required.

ii)
The first inclusion follows from Lemma~\ref{lemmasigmaappofTn1}~i).

Let $\lm\in \sigma_{\rm app}\big((T_n^*)_{n\in\N}\big)^*$.
Then there exist an infinite subset $I\subset\N$ and $x_n\in\dom(T_n^*)$, $n\in I$, with $\|x_n\|=1$ and $\|(T_n^*-\overline{\lm})\|\to 0$ as $n\to\infty$.
Since $(x_n)_{n\in I}$ is a bounded sequence and $H_0$ is weakly compact, there exists $\widetilde I\subset I$ such that $(x_n)_{n\in\widetilde I}$ converges weakly to some $x\in H_0$.
If $x=0$, then $\lm\in \sigma_{\rm ess}\left((T_n^*)_{n\in\N}\right)^*$.

Now assume that $x\neq 0$. Since $T_n\gsr T$, there exists $\lm_0\in\Delta_b\left((T_n)_{n\in\N}\right)\cap\rho(T)$ such that $(T_n-\lm_0)^{-1}P_n\s (T-\lm_0)^{-1}P$.
The convergence $\|(T_n^*-\overline{\lm})x_n\|\to 0$ implies $$(T_n^*-\lbar{\lm_0})x_n=(\overline{\lm}-\lbar{\lm_0})x_n+y_n,\quad\text{with}\quad y_n:=(T_n^*-\overline{\lm})x_n\tolong 0, \quad n\to\infty,$$
hence 
\begin{align*}
(T_n^*-\lbar{\lm_0})^{-1}x_n&=(\overline{\lm}-\lbar{\lm_0})^{-1}x_n-\widetilde y_n,\quad \widetilde y_n:=(\overline{\lm}-\lbar{\lm_0})^{-1} (T_n^*-\lbar{\lm_0})^{-1}y_n,\\
\left\|\widetilde y_n\right\|&\leq |{\lm}-{\lm_0}|^{-1} \big\|(T_n-{\lm_0})^{-1}\big\|\|y_n\|\tolong 0, \quad n\to\infty.
\end{align*}
Since $x_n\stackrel{w}{\to}x$, we obtain $(T_n^*-\lbar{\lm_0})^{-1}x_n\stackrel{w}{\to}(\overline{\lm}-\lbar{\lm_0})^{-1}x$ as $n\in\widetilde I$, $n\to\infty$. On the other hand, Lemma~\ref{lemmadisccompletelycont}~i) yields $x\in H$ and $(T_n^*-\lbar{\lm_0})^{-1}x_n\stackrel{w}{\to}(T^*-\lbar{\lm_0})^{-1}x$. 
By the uniqueness of the weak limit, we obtain $(T^*-\lbar{\lm_0})^{-1}x=(\overline{\lm}-\lbar{\lm_0})^{-1}x$, hence $(\overline{\lm}-\lbar{\lm_0})^{-1}\in\sigma_p((T^*-\lbar{\lm_0})^{-1})$.
This yields 
$\lm\in\sigma_p(T^*)^*$.

iii)
The second equality follows from claim~ii), from $\sigma_p(T^*)^*\subset\sigma_{\rm app}(T^*)^*$ and from claim i) (applied to $T^*,T_n^*$).
Now we obtain the first equality by replacing $T^*,T_n^*$ by $T,T_n$.
\end{proof}

The limiting essential spectrum is related to the region of boundedness as follows.
\begin{prop} \label{thmregionofbddvsseqessspectrum}
\begin{enumerate}[label=\rm{\roman{*})}] 
\item If $T_n\gsr T$ and $T_n^*\gsr T^*$, then \label{claimsigmainsigmapcupsigmaeofTnadj} 
\begin{align*}
\C\backslash\Delta_b\left((T_n)_{n\in\N}\right)
&=\sigma_p(T)\cup\sigma_{\rm ess}\left((T_n)_{n\in\N}\right)\cup\sigma_p(T^*)^*\cup\sigma_{\rm ess}\left((T_n^*)_{n\in\N}\right)^*,\\
\Delta_b\big((T_n)_{n\in\N}\big)\cap\rho(T)&=\big(\C\backslash(\sigma_{\rm ess}\left((T_n)_{n\in\N}\right)\cup\sigma_{\rm ess}\left((T_n^*)_{n\in\N}\right)^*)\big)\cap\rho(T).
\end{align*}

\item If $T_n\gsr T$ and $T_n,\,n\in\N,$ all have compact resolvents, then 
\begin{align*}
\begin{aligned}
\C\backslash\Delta_b\left((T_n)_{n\in\N}\right)&\subset\sigma_p(T^*)^*\cup\sigma_{\rm ess}\left((T_n^*)_{n\in\N}\right)^*,\\
\quad \Delta_b\big((T_n)_{n\in\N}\big)\cap\rho(T)&=\big(\C\backslash\sigma_{\rm ess}\big((T_n^*)_{n\in\N}\big)^*\big)\cap\rho(T).\label{eqclaimcompres}
\end{aligned}
\end{align*}
\end{enumerate}
\end{prop}

\begin{proof}
i)
The claimed identities follow from Lemma~\ref{lemmasigmaappofTn1}~ii) and  Proposition~\ref{lemmasigmaappofTn}~iii).

ii)
 The claims follow from  the second part of Lemma~\ref{lemmasigmaappofTn1}~ii) and  Proposition~\ref{lemmasigmaappofTn}~ii).
\end{proof}

The local spectral convergence result (Theorem~\ref{mainthmspectralexactness}) relies on the following result from~\cite{Boegli-chap1}.

\begin{theorem} {\rm \cite[Theorem~2.3]{Boegli-chap1}}\label{thmlocal}
Suppose that $T_n\gsr T$.
\begin{enumerate}
\item[\rm i)] For each $\lm\in\sigma(T)$ such that for some $\eps>0$ we have 
$$ B_{\eps}(\lm)\backslash\{\lm\}\subset \Delta_b\left((T_n)_{n\in\N}\right)\cap\rho(T),$$  
there exist $\lm_n\in\sigma(T_n), \,n\in\N,$ with $\lm_n\to\lm$ as $n\to\infty$.

\item[\rm ii)] No spectral pollution occurs in $\Delta_b\left((T_n)_{n\in\N}\right)$.
\end{enumerate}
\end{theorem}

\begin{proof}[Proof of Theorem~{\rm\ref{mainthmspectralexactness}}]
i)
First note that the set in~\eqref{eqbadset} is closed by Proposition~\ref{propsigmaeclosed}.
If $\lm$ is an isolated point of $\sigma(T)$ and does not belong to the set in~\eqref{eqbadset}, then
there exists $\eps>0$ so small that 
$$B_{\eps}(\lm)\backslash\{\lm\}\subset\Big(\C\backslash\big(\sigma_{\rm ess}\left((T_n)_{n\in\N}\right)\cup\sigma_{\rm ess}\left((T_n^*)_{n\in\N}\right)^*\big)\Big)\cap\rho(T).$$
By Proposition~\ref{thmregionofbddvsseqessspectrum}~i), the right hand side coincides with $\Delta_b\big((T_n)_{n\in\N}\big)\cap\rho(T).$
Now the  claims follow from  Theorem~\ref{thmlocal}.

ii) The proof is analogous to i); we use 
claim ii) of Proposition~\ref{thmregionofbddvsseqessspectrum}.

iii)
Assume that the claim is false. Then there exist $\alpha>0$, an infinite subset $I\subset\N$ and $\lm_n\in K$, $n\in I$, such that one of the following holds:
\begin{enumerate}
\item[\rm(1)]  $\lm_n\in\sigma(T_n)$ and ${\rm dist}(\lm_n,\sigma(T)\cap K)>\alpha$ for every $n\in I$;
\item[\rm(2)]  $\lm_n\in\sigma(T)$ and ${\rm dist}(\lm_n,\sigma(T_n)\cap K)>\alpha$ for every $n\in I$.
\end{enumerate}
Note that, in both cases (1) and (2), the compactness of $K$ implies that there exist $\lm\in K$ and an infinite subset $J\subset I$ such that $(\lm_n)_{n\in J}$ converges to $\lm$.

First we consider case (1).
There are $ \lm_n\in\sigma(T_n)$, $n\in J$, with $\lm_n\to\lm\in K$. Since $K$ does not contain spectral pollution by the assumptions, we conclude $\lm\in\sigma(T)\cap K$. Hence 
$$ |\lm_n-\lm|\geq {\rm dist}(\lm_n, \sigma(T)\cap K)>\alpha, \quad n\in J,$$
 a contradiction to $\lm_n\to\lm$.

Now assume that (2) holds. The closedness of $\sigma(T)\cap K$ yields $\lm\in\sigma(T)\cap K$, and the latter set is discrete by the assumptions. So there exists $n_0\in\N$ so that $\lm=\lm_n$ for all $n\in J$ with $n\geq n_0$.
In addition, by the above claim i) or ii), respectively,  there exist $\mu_n\in\sigma(T_n)$, $n\in\N$, so that $\mu_n\to \lm$ as $n\to\infty$. Since $\lm\in\sigma(T)\cap K$ is in the interior of $K$ by the assumptions, there exists $n_1\in\N$ so that $\mu_n\in K$ for all $n\geq n_1$.
So we conclude that, for all $n\in J$ with $n\geq \max\{n_0,n_1\}$,
$$|\lm-\mu_n|\geq {\rm dist}(\lm, \sigma(T_n)\cap K))={\rm dist}(\lm_n, \sigma(T_n)\cap K)>\alpha,$$
a contradiction to $\mu_n\to\lm$.
This proves the claim.
\end{proof}

We illustrate the latter result for the Galerkin method of a compact perturbation of a Toeplitz operator.
It is well known that truncating a Toeplitz operator (and compact perturbations of it) to finite sections is not a spectrally exact process but the pseudospectra converge (see~\cite[Chapters 2, 3]{Boettcher-Silbermann} for an overview).
We confirm these results using the developed theory of limiting essential spectra.

\begin{example}
Denote by $\{e_k:\,k\in\N\}$ the standard orthonormal basis of $l^2(\N)$.
Let $T\in L(l^2(\N))$ be the \emph{Toeplitz operator}
defined by the so-called \emph{symbol} $$f(z):=\sum_{k\in\Z} a_k z^k, \quad z\in\C,$$
 where $a_k\in\C$, $k\in\Z$, are chosen so that $f$ is continuous.
This means that, with respect to $\{e_k:\,k\in\N\}$, the operator $T$ has the matrix representation $(T_{ij})_{i,j=1}^{\infty}$ with
$$T_{ij}:=\langle Te_j,e_i\rangle=a_{i-j}, \quad i,j\in\N.$$ 
The set  $f(\partial B_1(0))$ is called \emph{symbol curve}.
Given $\lm\notin f(\partial B_1(0))$, we define the \emph{winding number} $I(f,\lm)$  to be the winding number of $f(\partial B_1(0))$ about $\lm$ in the usual positive (counterclockwise) sense.
The spectrum of $T$ is, by~\cite[Theorem~1.17]{Boettcher-Silbermann}, given by $$\sigma(T)=f(\partial B_1(0))\cup\big\{\lm\notin f(\partial B_1(0)):\, I(f,\lm)\neq 0\big\}.$$

For $n\in\N$, let $P_n$ be the orthogonal projection of $l^2(\N)$ onto $H_n:={\rm span}\{e_k:\,k=1,\dots,n\}$. It is easy to see that $P_n\s I$.
For a compact operator $S\in L(l^2(\N))$, let $A:=T+S$ and define $A_n:=P_nA|_{H_n}$, $n\in\N$. 
We claim that the limiting essential spectra satisfy 
\beq\label{eq.toeplitz}
\sigma_{\rm ess}\big((A_n)_{n\in\N}\big)\cup \sigma_{\rm ess}\big((A_n^*)_{n\in\N}\big)^*\subset \sigma(T)\subset\sigma(A);
\eeq
hence, by Theorem~\ref{mainthmspectralexactness}, no spectral pollution occurs for the approximation $(A_n)_{n\in\N}$ of~$A$, and every isolated $\lm\in\sigma(A)\backslash\sigma(T)$ is the limit of a sequence $(\lm_n)_{n\in\N}$ with $\lm_n\in\sigma(A_n)$, $n\in\N$.

To prove these statements, define  $T_n:=P_nT|_{H_n}$, $n\in\N$.  Clearly, $T_nP_n\s T$, $A_nP_n\s A$ and $T_n^*P_n\s T^*$, $A_n^*P_n\s A^*$. 
Hence $T_n\gsr T$, $A_n\gsr A$ and $T_n^*\gsr T^*$, $A_n^*\gsr A^*$.
By~\cite[Theorem~2.11]{Boettcher-Silbermann}, $\rho(T)\subset\Delta_b\left((T_n)_{n\in\N}\right)$.
Using Proposition~\ref{thmregionofbddvsseqessspectrum}~i), we obtain
$$\sigma_{\rm ess}\big((T_n)_{n\in\N}\big)\cup \sigma_{\rm ess}\big((T_n^*)_{n\in\N}\big)^*\subset\C\backslash\Delta_b\big((T_n)_{n\in\N}\big).$$
The perturbation result in Theorem~\ref{thmreldisccomppert}~i) implies 
$$\sigma_{\rm ess}\big((A_n)_{n\in\N}\big)\cup \sigma_{\rm ess}\big((A_n^*)_{n\in\N}\big)^*=\sigma_{\rm ess}\big((T_n)_{n\in\N}\big)\cup \sigma_{\rm ess}\big((T_n^*)_{n\in\N}\big)^*.$$
So, altogether we arrive at the first inclusion in~\eqref{eq.toeplitz}.
By~\cite[Theorem~1.17]{Boettcher-Silbermann}, $\sigma_{\rm ess}(T)\cup\sigma_{\rm ess}(T^*)^*=f(\partial B_1(0))$ and for $\lm\in\sigma(T)\backslash f(\partial B_1(0))$ the operator $T-\lm$ is Fredholm with index ${\rm ind}(T-\lm)=-I(f,\lm)\neq 0$.
This means that $\sigma(T)$ is equal to the set $\sigma_{e4}(T)$ defined in~\cite[Chapter~IX]{edmundsevans}, one of the (in general not equivalent) characterisations of essential spectrum. This set is invariant under compact perturbations by~\cite[Theorem~IX.2.1]{edmundsevans}, hence $\sigma_{e4}(T)=\sigma_{e4}(A)\subset\sigma(A)$,
 which proves the second inclusion  in~\eqref{eq.toeplitz}.
The rest of the claim follows from Theorem~\ref{mainthmspectralexactness}.

For a concrete example, let
\begin{align*}
a_{-3}&=-7, \quad a_{-2}=8, \quad a_{-1}=-1,\quad a_2=15, \quad a_3= 5, \\ a_k&=0, \quad k\in\Z\backslash\{-3,-2,-1,2,3\}.
 \end{align*}

\begin{figure}[htb]
\begin{center}
\subfigure[Symbol curve (red) corresponding to $T$ and winding number in each component.]{\includegraphics[width=0.55\textwidth]{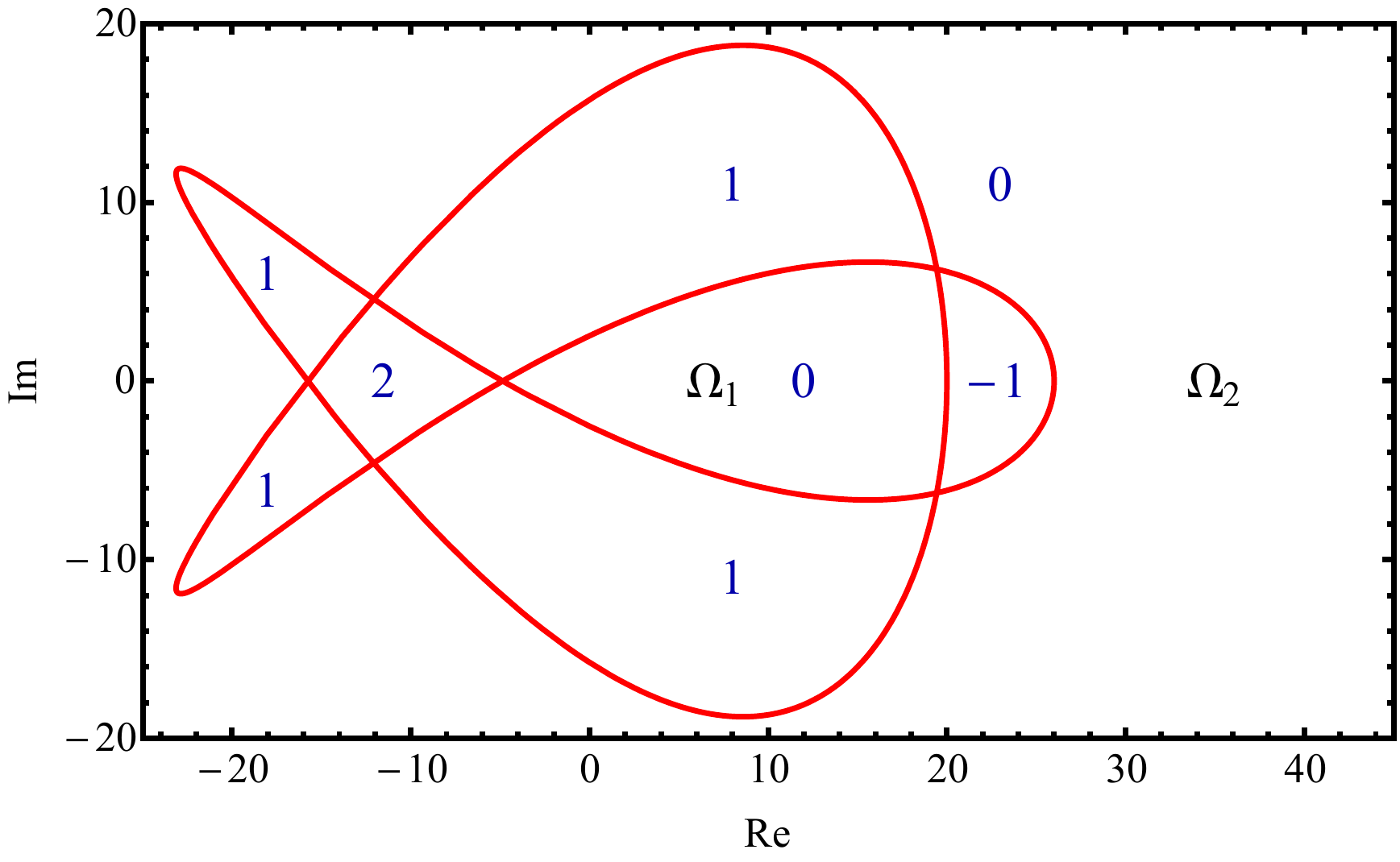}}\\[1mm]
 \includegraphics[width=0.55\textwidth]{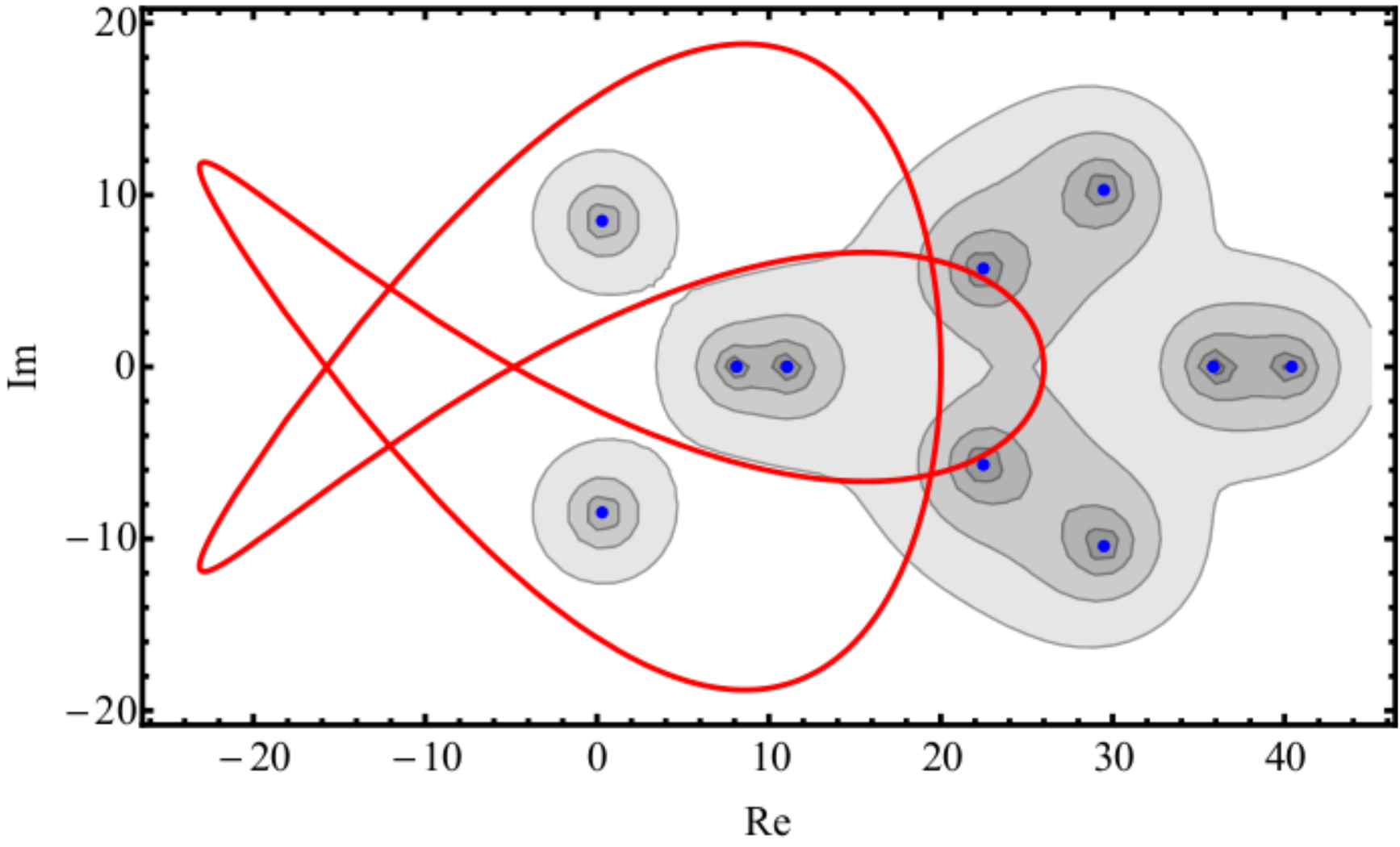}\\
 \includegraphics[width=0.55\textwidth]{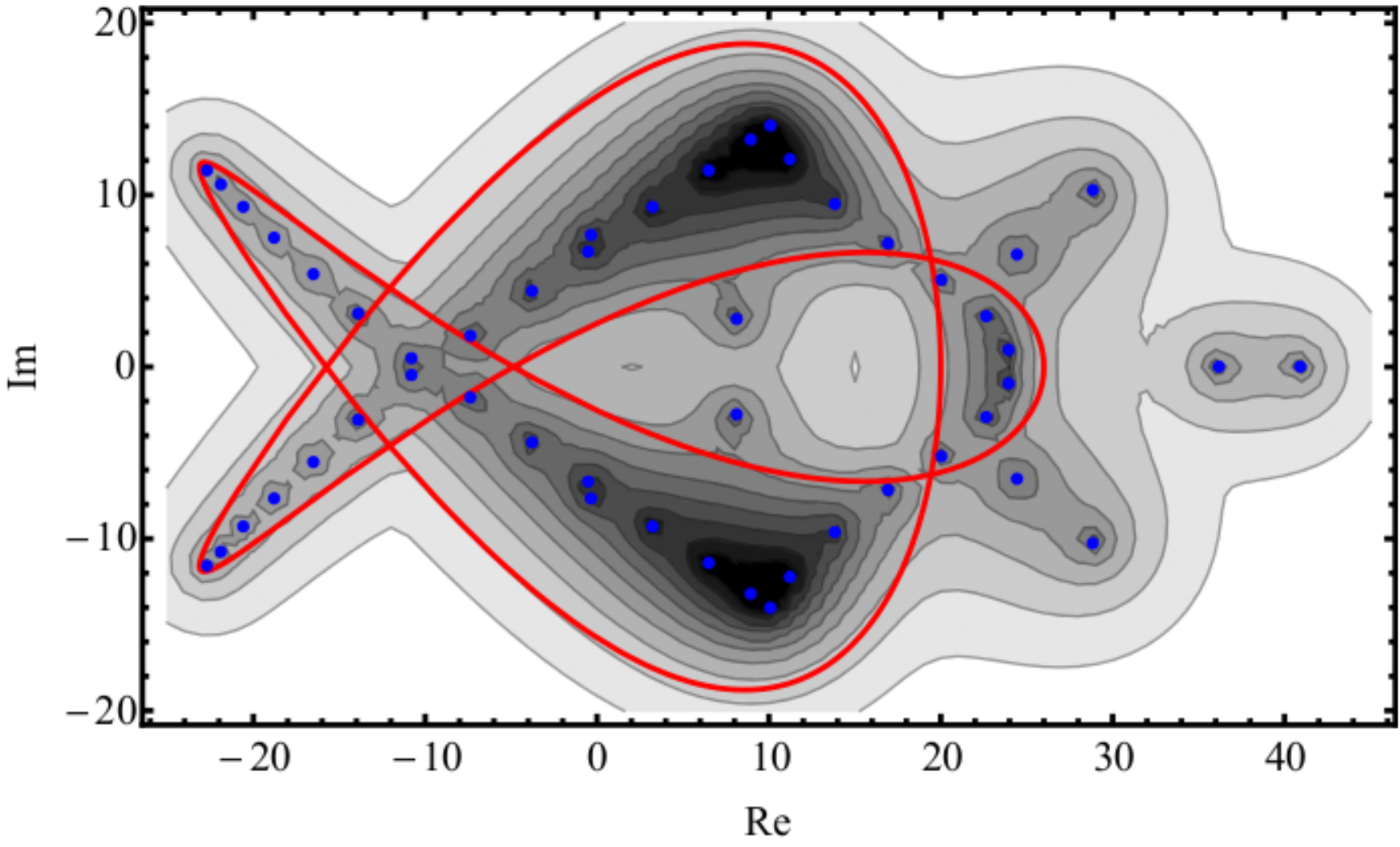}\\
\subfigure[Eigenvalues (blue dots) of $A_n$ for $n=10$ (top), $n=50$ (middle), $n=100$ (bottom) and $\eps$-pseudospectra of $A_n$ for $\eps=2,1,2^{-1},\dots,2^{-5}$.]{
 \includegraphics[width=0.55\textwidth]{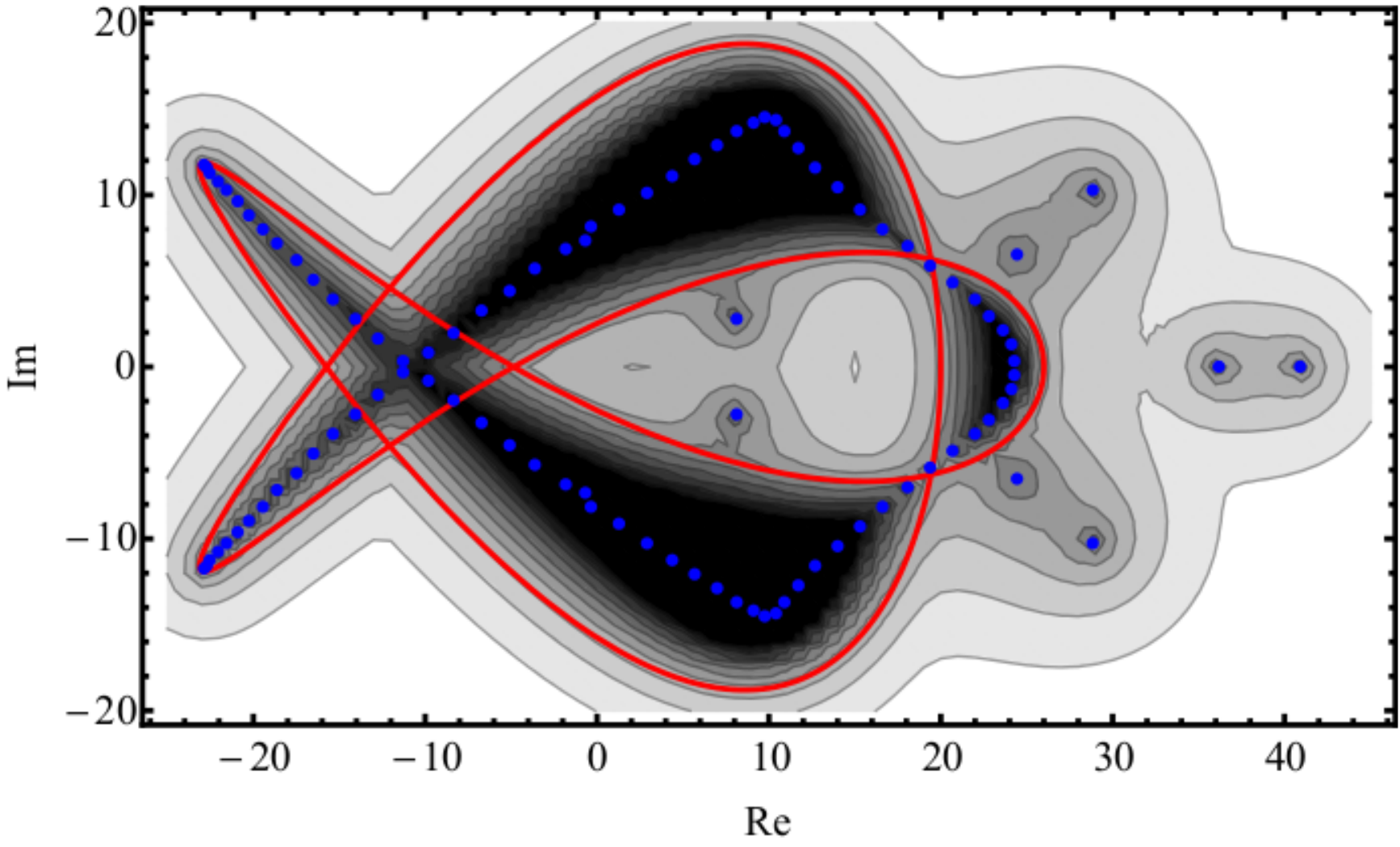}
}
\caption{\small Spectra and pseudospectra of the truncated $n\times n$ matrices $A_n$ of the perturbed Toeplitz operator $A$.}
\label{figsymbolcurvefish}
\end{center}
\end{figure}

\noindent The symbol curve $f(\partial B_1(0))$ is shown in Figure~\ref{figsymbolcurvefish} in red. 
The spectrum of the corresponding Toeplitz operator $T$ consists of the symbol curve together with the connected components with winding numbers $1,2,-1$, see Figure~\ref{figsymbolcurvefish}~(a).
So the resolvent set $\rho(T)$ is the union of the  connected components with winding number~$0$, which are denoted by $\Omega_1$ and $\Omega_2$ in Figure~\ref{figsymbolcurvefish}~(a). 

Now we add the compact operator $S$ with matrix representation $$(S_{ij})_{i,j=1}^{\infty}, \quad S_{ij}:=\begin{cases} 20, & i=j\leq 10,\\ 0 , &\text{otherwise}.\end{cases}$$
By the above claim, the Galerkin approximation $(A_n)_{n\in\N}$ of $A:=T+S$ does not produce spectral pollution, and every accumulation point of $\sigma(A_n)$, $n\in\N$, in $\Omega_1\cup\Omega_2$ belongs to $\sigma(A)$.  Figure~\ref{figsymbolcurvefish}~(b) suggests that two such accumulation points exist in $\Omega_1$ and six in $\Omega_2$.

Note that although the points in $\sigma(T)\subset\sigma(A)$ are not approximated by the Galerkin method, the resolvent norm diverges at these points, see Figure~\ref{figsymbolcurvefish}~(b). This is justified by~\cite[Proposition~4.2]{Boettcher-Wolf-1997}, which implies that for every $\lm\in\sigma(A)$ and every $\eps>0$ there exists $n_{\lm,\eps}\in\N$ with $\lm\in\sigma_{\eps}(A_n)$, $n\geq n_{\lm,\eps}$ (see also Theorem~\ref{thm.inclusion.prelim} below).
\end{example}

\section{Local convergence of  pseudospectra}\label{sectionpseudospectra}
In this section we establish special properties and convergence of pseudospectra.  Subsection~\ref{subsectionmainresults.pseudo} contains the main pseudospectral convergence result (Theorem~\ref{thmexactness2}). We also study the special case of operators having constant resolvent norm on an open set (Theorem~\ref{thmconstresnorm}). 
In Subsection~\ref{subsectionpropertiessigmaess.pseudo} we provide properties of the limiting essential $\eps$-near spectrum $\Lambda_{{\rm ess}, \eps}\big((T_n)_{n\in\N}\big)$ including a perturbation result (Theorem~\ref{prop.pert.pseudo}), followed by Subsection~\ref{subsectionproofofmainresults.pseudo} with the proofs of the results stated in Subsection~\ref{subsectionmainresults.pseudo}.

\subsection{Main convergence result}\label{subsectionmainresults.pseudo}

We fix an $\eps>0$. 

\begin{definition}\label{defepspseudo}
Define the 
\emph{$\eps$-approximate point spectrum} of $T$  by
\begin{align*}
\sigma_{{\rm app},\eps}(T)
&:=\big\{\lm\in\C:\,\exists\, x\in\dom(T)\,\text{with}\,\|x\|=1,\,\|(T-\lm)x\|<\eps\big\}.
\end{align*}
\end{definition}

The following properties are well-known, see for instance \cite[Chapter 4]{Trefethen-2005} and~\cite{Chatelin-Harrabi}.
\begin{lemma}\label{lemma.eps.T}
\begin{enumerate}
\item[\rm i)]
The sets $\sigma_{\eps}(T)$,  $\sigma_{{\rm app},\eps}(T)$ are open  subsets of~$\C$.
\item[\rm ii)]
We have 
\begin{align*}
\sigma_{\eps}(T)&=\sigma_{{\rm app},\eps}(T)\cup\sigma(T)
=\sigma_{{\rm app},\eps}(T)\cup \sigma_{{\rm app},\eps}(T^*)^*,
\end{align*}
and 
$$\sigma_{{\rm app},\eps}(T)\backslash\sigma(T)=\sigma_{{\rm app},\eps}(T^*)^*\backslash\sigma(T).$$
If $T$ has compact resolvent, then 
$$\sigma_{\eps}(T)=\sigma_{{\rm app},\eps}(T)=\sigma_{{\rm app},\eps}(T^*)^*.$$
\item[\rm iii)]
For $\eps>\eps'>0$,
\begin{align*}
\sigma_{\eps'}(T) &\subset \sigma_{\eps}(T), \quad
\underset{\eps>0}{\bigcap}\sigma_{\eps}(T)=\sigma(T).
\end{align*}

\item[\rm iv)]
We have 
\begin{align*}
\big\{\lm\in\C:\,{\rm dist}\big(\lm,\sigma(T)\big)< \eps\big\}&\subset \sigma_{\eps}(T),
\end{align*}
with equality if $T$ is selfadjoint.
\end{enumerate}
\end{lemma}

In contrast to the spectrum, the $\eps$-pseudospectrum is always approximated under generalised strong resolvent convergence.
For bounded operators and strong convergence, this was proved by B\"ottcher-Wolf in~\cite[Proposition~4.2]{Boettcher-Wolf-1997} (where non-strict inequality in the definition of pseudospectra is used);
claim~i) is not explicitly stated but can be read off from the proof.
Note that if $T$ has compact resolvent, then claim i) holds for all $\lm\in\sigma_{\eps}(T)=\sigma_{{\rm app},\eps}(T)$ by Lemma~\ref{lemma.eps.T}~ii).

\begin{theorem}\label{thm.inclusion.prelim}
Suppose that $T_n\gsr T$. 
\begin{enumerate}
\item[\rm i)]
For every $\lm\in\sigma_{{\rm app},\eps}(T)$ and $x\in\dom(T)$ with $\|x\|=1$, $\|(T-\lm)x\|<\eps$ there exist $n_{\lm}\in\N$ and $x_n\in\dom(T_n)$, $n\geq n_{\lm}$, with
$$\lm\in\sigma_{{\rm app},\eps}(T_n),\quad \|x_n\|=1,\quad \|(T_n-\lm)x_n\|<\eps, \quad n\geq n_{\lm},$$
and $\|x_n-x\|\to 0$ as $n\to\infty$.
\item[\rm ii)]
Suppose that, in addition, $T_n^*\gsr T^*$. Then for every $\lm\in\sigma_{\eps}(T)$ there exists $n_{\lm}\in\N$ such that $\lm\in\sigma_{\eps}(T_n)$, $n\geq n_{\lm}$. 
\end{enumerate}
\end{theorem}

The following example illustrates that we cannot omit the additional assumption  $T_n^*\gsr T^*$ in Theorem~\ref{thm.inclusion.prelim}~ii).
In particular, this is a counterexample for~\cite[Theorem~4.4]{Brown-Marletta-2001} where only $T_n\gsr T$ is assumed.

\begin{example}\label{ex.firstderivative}
Let $T$ be the first derivative in $L^2(0,\infty)$ with Dirichlet boundary condition,
$$Tf:=f', \quad \dom(T):=\{f\in W^{1,2}(0,\infty):\,f(0)=0\}.$$
We approximate $T$ by a sequence of operators $T_n$ in $L^2(0,n)$, $n\in\N$, defined by
$$T_nf:=f', \quad \dom(T_n):=\{f\in W^{1,2}(0,n):\,f(0)=f(n)\}.$$
Note that $\{\lm\in\C:\,\re\,\lm\geq 0\}=\sigma(T)\subset\sigma_{\eps}(T).$
The operators $\I\,T_n$, $n\in\N$, are selfadjoint. Hence 
$$\{\lm\in\C:\,\re\,\lm<0\}\subset\Delta_b\big((T_n)_{n\in\N}\big)\cap\rho(T).$$
Using that $\{f\in\dom(T):\,\supp f\, \text{compact}\}$ is a core of $T$, \cite[Theorem~3.1]{Boegli-chap1} implies that $T_n\gsr T$.
However, since $\I\,T$ is not selfadjoint, we obtain $T_n^*=-T_n \gsr - T\neq  T^*$.
The selfadjointness of $\I\,T_n$ and Lemma~\ref{lemma.eps.T}~iv) imply$$\sigma_{\eps}(T_n)=\{\lm\in\C:\,{\rm dist}(\lm,\sigma(T_n))<\eps\}\subset\{\lm\in\C:\,|\re\,\lm|<\eps\}, \quad n\in\N.$$
Therefore, for every $\lm\in\C$ with $\re\,\lm>\eps$, we conclude $\lm\in\sigma_{\eps}(T)$ but
$${\rm dist}(\lm,\sigma_{\eps}(T_n))\geq \re\,\lm-\eps>0, \quad n\in\N.$$
\end{example}

In order to characterise $\eps$-pseudospectral pollution, we introduce the following sets.

\begin{definition}\label{def.limiting.eps}
Define the
\emph{essential $\eps$-near spectrum} of $T$ by
\begin{align*}
\Lambda_{{\rm ess},\eps}(T)
&:=\left\{\lm\in\C:\,\exists\, x_n\in\dom(T),\,n\in\N,\,\text{with}\,\begin{array}{l}\|x_n\|=1,\,x_n\stackrel{w}{\to}0,\\
\|(T-\lm)x_n\|\to \eps\end{array}\right\},
\end{align*}
and the 
 \emph{limiting essential $\eps$-near spectrum} of $(T_n)_{n\in\N}$ by
\begin{align*}
\Lambda_{{\rm ess}, \eps}\big((T_n)_{n\in\N}\big):=\left\{\lm\in\C:\,
\begin{array}{l} \exists\,I\subset\N\,\exists\,x_n\in\dom(T_n),\,n\in I,\,\text{with}\\
\|x_n\|=1,\,x_n\stackrel{w}{\to}0,\,\|(T_n-\lm)x_n\|\to \eps
\end{array}
\right\}.
\end{align*}
\end{definition}

The following theorem is the main result of this section. We establish local $\eps$-pseudospectral exactness and prove $\eps$-pseudospectral convergence with respect to the Hausdorff metric in compact subsets of the complex plane where we have $\eps$-pseudospectral exactness.

\begin{theorem}\label{thmexactness2}
Suppose that $T_n\gsr T$ and $T_n^*\gsr T^*$.
\begin{enumerate}
\item[\rm i)] 
The sequence $(T_n)_{n\in\N}$ is an $\eps$-pseudospectrally inclusive approximation of~$T$. 

\item[\rm ii)]
Define \vspace{-1mm}
$$\Lambda_{{\rm ess},(0,\eps]}:=\underset{\delta\in (0,\eps]}{\bigcup}\Big(\Lambda_{{\rm ess},\delta}\big((T_n)_{n\in\N}\big)\cap\Lambda_{{\rm ess},\delta}\big((T_n^*)_{n\in\N}\big)^*\Big).$$
Then $\eps$-pseudospectral pollution is confined to 
\beq\label{eq.defKgeneral}
\sigma_{\rm ess}\big((T_n)_{n\in\N}\big)\cup \sigma_{\rm ess}\big((T_n^*)_{n\in\N}\big)^*\cup \Lambda_{{\rm ess},(0,\eps]};
\eeq
if the operators $T_n$, $n\in\N$, all have compact resolvents, then it is restricted to 
\beq \label{eq.defK2}
\big(\sigma_{\rm ess}\big((T_n)_{n\in\N}\big)\cap \sigma_{\rm ess}\big((T_n^*)_{n\in\N}\big)^*\big)\cup \Lambda_{{\rm ess},(0,\eps]}.
\eeq

\item[\rm iii)]
Let $K\subset \C$ be a compact subset with
$$\overline{\sigma_{\eps}(T)}\cap K=\overline{\sigma_{\eps}(T)\cap K}\neq \emptyset.$$
If  the intersection of $K$ with the set in~\eqref{eq.defKgeneral} or~\eqref{eq.defK2}, respectively, is contained in $\overline{\sigma_{\eps}(T)}$, then
$${\rm d_H}\big(\overline{\sigma_{\eps}(T_n)}\cap K,\overline{\sigma_{\eps}(T)}\cap K\big)\tolong 0, \quad n\to\infty.$$
\end{enumerate}
\end{theorem}

\begin{rem}
If we compare Theorem~\ref{thmexactness2}~iii) with \cite[Theorem~2.1]{Boegli-2014-80} for generalised \emph{norm} resolvent convergence, note that here we do no explicitly exclude the possibility that $\lm\mapsto\|(T-\lm)^{-1}\|$ is constant on an open subset $\emptyset\neq U\subset\rho(T)$. However,  if the resolvent norm is equal to $1/\eps$ on an open set $U$, then  $U\cap\overline{\sigma_{\eps}(T)}=\emptyset$ and hence the following Theorem~\ref{thmconstresnorm}~ii) implies that a compact set $K$ with $K\cap\Lambda_{{\rm ess},\eps}\big((T_n)_{n\in\N}\big)\subset\overline{\sigma_{\eps}(T)}$ satisfies $K\cap U=\emptyset$.
So we implicitly exclude the problematic region $U$.
\end{rem}

In the following result we study operators that have constant resolvent norm on an open set.
For the existence of such operators see~\cite{Shargorodsky-2008-40,Boegli-2014-80}.

\begin{theorem}\label{thmconstresnorm}
Assume that there exists an open subset $\emptyset\neq U\subset\rho(T)$ such that
$$\|(T-\lm)^{-1}\|=\frac{1}{\eps},\quad \lm\in U.$$
\begin{enumerate}
\item[\rm i)]
We have
\begin{align*}
 \rho(T)\subset \C\backslash \underset{K \text{ compact}}{\bigcap}\sigma(T+K)\subset  \Lambda_{{\rm ess},\eps}(T)\cap \Lambda_{{\rm ess},\eps}(T^*)^*.
\end{align*}

\item[\rm ii)]
If $T_n\gsr T$ and $T_n^*\gsr T^*$, then 
$$\rho(T)\subset \C\backslash \underset{K \text{ compact}}{\bigcap}\sigma(T+K)\subset \Lambda_{{\rm ess},\eps}\big((T_n)_{n\in\N}\big)\cap \Lambda_{{\rm ess},\eps}\big((T_n^*)_{n\in\N}\big)^*.$$
\end{enumerate}
\end{theorem}

\begin{rem}
Note that, by~\cite[Section~IX.1, Theorems~IX.1.3,~1.4]{edmundsevans}, 
$$\sigma_{\rm ess}(T)\cup\sigma_{\rm ess}(T^*)^*=\sigma_{e3}(T)\subset\sigma_{e4}(T)=\underset{K \text{ compact}}{\bigcap}\sigma(T+K).$$
\end{rem}

\subsection{Properties of the limiting essential $\eps$-near spectrum}\label{subsectionpropertiessigmaess.pseudo}


\begin{prop}\label{lemma.limiting.eps}
\begin{enumerate}
\item[\rm i)]
The sets $\Lambda_{{\rm ess},\eps}(T)$, $\Lambda_{{\rm ess},\eps}\big((T_n)_{n\in\N}\big)$  are closed subsets of~$\C$.

\item[\rm ii)]
We have 
\beq \label{eq.inclusionepsnbhofsigmaess}
\begin{aligned}
\big\{\lm+z:\,\lm\in\sigma_{\rm ess}(T),\,|z|=\eps\big\}&\subset \Lambda_{{\rm ess},\eps}(T),\\
\big\{\lm+z:\,\lm\in\sigma_{\rm ess}\big((T_n)_{n\in\N}\big),\,|z|=\eps\big\}&\subset \Lambda_{{\rm ess},\eps}\big((T_n)_{n\in\N}\big).
\end{aligned}
\eeq
\end{enumerate}

\end{prop}

\begin{proof}
A diagonal sequence argument implies claim~i), and claim ii) is easy to see.
\end{proof}

\begin{rem}
The inclusions in claim~ii) may be strict. 
In fact, for Shargorodsky's example \cite[Theorem~3.2]{Shargorodsky-2008-40} of an operator $T$ with constant ($1/\eps=1$) resolvent norm on an open set, the compressions $T_n$ onto the span of the first $2n$ basis vectors satisfy 
$$\sigma_{\rm ess}\big((T_n)_{n\in\N}\big)=\sigma_{\rm ess}(T)=\emptyset, \quad \underset{K \text{ compact}}{\bigcap}\sigma(T+K)=\emptyset.$$
Hence the left hand side of~\eqref{eq.inclusionepsnbhofsigmaess} is empty whereas the right hand side equals $\C$ by Theorem~\ref{thmconstresnorm}.
\end{rem}

Analogously as $\sigma_{\rm ess}(T)\subset\sigma_{\rm ess}\big((T_n)_{n\in\N}\big)$ (see Proposition~\ref{propsigmaess}), also the essential $\eps$-near spectrum is contained in its limiting counterpart.

\begin{prop}\label{proplimsigmaess}
\begin{enumerate}
\item[\rm i)] Assume that $T_n\gsr T$. Then $\Lambda_{{\rm ess},\eps}(T)\subset\Lambda_{{\rm ess},\eps}\big((T_n)_{n\in\N}\big).$
\item[\rm ii)] If $H_0=H$ and $T_n\gnr T$, then $\Lambda_{{\rm ess},\eps}(T)\cap\rho(T)=\Lambda_{{\rm ess},\eps}\big((T_n)_{n\in\N}\big)\cap\rho(T).$
\end{enumerate}
\end{prop}

For the proof we use the following simple result.
\begin{lemma}\label{lemmaClaim}
Assume that $T_n^*\gsr T^*$. Suppose that there exist an infinite subset $I\subset\N$ and $x_n\in\dom(T_n)$, $n\in I$, with $\|x_n\|=1$ and $x_n\stackrel{w}{\to} 0$. If $(\|T_nx_n\|)_{n\in I}$ is bounded, then $T_nx_n\stackrel{w}{\to} 0$.
\end{lemma}

\begin{proof}
Define $y_n:=T_nx_n$. Let $\lm_0\in\Delta_b\big((T_n)_{n\in\N}\cap\rho(T)$ satisfy $(T_n^*-\overline{\lm_0})^{-1}P_n\s (T^*-\overline{\lm_0})^{-1}P$.
Since $H_0$ is weakly compact, there exist $y\in H_0$ and an infinite subset $\widetilde I\subset I$ such that $(y_n)_{n\in\widetilde I}$ converges weakly to $y$. We prove that $y=0$.
Lemma~\ref{lemmadisccompletelycont}~i) implies $y\in H$ and $(T_n-\lm_0)^{-1}y_n\stackrel{w}{\to}(T-\lm_0)^{-1}y$.
Hence $x_n=(T_n-\lm_0)^{-1}y_n-\lm_0x_n\stackrel{w}{\to} (T-\lm_0)^{-1}y$.
The uniqueness of the weak limit yields $(T-\lm_0)^{-1}y=0$ and thus $y=0$.
\end{proof}

\begin{proof}[Proof of Proposition~{\rm\ref{proplimsigmaess}}]
i) The proof is analogous to the one of Proposition~\ref{propsigmaess}~i).

ii) Using claim i), it remains to prove $\Lambda_{{\rm ess},\eps}\big((T_n)_{n\in\N}\big)\cap\rho(T)\subset \Lambda_{{\rm ess},\eps}(T)$.
Let $\lm\in \Lambda_{{\rm ess},\eps}\big((T_n)_{n\in\N}\big)\cap\rho(T)$. 
By Definition~\ref{def.limiting.eps}, there exist an infinite subset $I\subset\N$ and $x_n\in\dom(T_n)$, $n\in I$, with $\|x_n\|=1$, $x_n\stackrel{w}{\to} 0$ and $\|(T_n-\lm)x_n\|\to\eps$.
Since $\lm\in\rho(T)$,~\cite[Proposition~2.16~ii)]{Boegli-chap1} implies that there exists $n_{\lm}\in\N$ such that $\lm\in\rho(T_n)$, $n\geq n_{\lm}$, and $(T_n-\lm)^{-1}P_n\to (T-\lm)^{-1}$. Define $I_2:=\{\lm\in I:\,n\geq n_{\lm}\}$ and
 $$w_n:=\frac{(T_n-\lm)x_n}{\|(T_n-\lm)x_n\|}\in H_n\subset H, \quad n\in I_2.$$
Then $\|w_n\|=1$ and $w_n\stackrel{w}{\to} 0$ by Lemma~\ref{lemmaClaim}. 
In addition,
$$\|(T_n-\lm)^{-1}w_n\|=\frac{1}{\|(T_n-\lm)x_n\|}\tolong \frac{1}{\eps}, \quad n\in I_2, \quad n\to\infty.$$
Since, in the limit $n\in I_2$, $n\to\infty,$
$$\big|\|(T_n-\lm)^{-1}w_n\|-\|(T-\lm)^{-1}w_n\|\big|\leq \|(T_n-\lm)^{-1}P_n-(T-\lm)^{-1}\|\tolong 0,$$
we conclude $\|(T-\lm)^{-1}w_n\|\to 1/\eps$.
Now define 
$$v_n:=\frac{(T-\lm)^{-1}w_n}{\|(T-\lm)^{-1}w_n\|}\in\dom(T), \quad n\in  I_2.$$
Then $\|v_n\|=1$ and $v_n\stackrel{w}{\to}0$.
Moreover, $\|(T-\lm)v_n\|=\|(T-\lm)^{-1}w_n\|^{-1}\to \eps$, hence $\lm\in\Lambda_{{\rm ess},\eps}(T)$.
\end{proof}

Similarly as for $\sigma_{\rm ess}\big((T_n)_{n\in\N}\big)$ (see Proposition~\ref{propdisccomp}), the set $\Lambda_{{\rm ess},\eps}\big((T_n)_{n\in\N}\big)$ is particularly simple if the operators $T_n$, $n\in\N$, or their resolvents form a discretely compact sequence. 

\begin{prop}\label{propdisccompeps}
\begin{enumerate}
\item[\rm i)]
If $T_n\in L(H_n)$, $n\in\N$, are so that $(T_n)_{n\in\N}$ is a discretely compact sequence and $(T_n^*P_n)_{n\in\N}$ is strongly convergent, then $$\Lambda_{{\rm ess},\eps}\big((T_n)_{n\in\N}\big)=\{\lm\in\C:\,|\lm|=\eps\}.$$
If, in addition, $(T_nP_n)_{n\in\N}$ is strongly convergent, then $$\Lambda_{{\rm ess},\eps}\big((T_n^*)_{n\in\N}\big)^*=\{\lm\in\C:\,|\lm|=\eps\}.$$

\item[\rm ii)]
If there exists $\lm_0\in\underset{n\in\N}{\bigcap}\rho(T_n)\cap\rho(T)$ such that $((T_n-\lm_0)^{-1})_{n\in\N}$ is a discretely compact sequence and $(T_n^*-\overline{\lm_0})^{-1}P_n\s (T^*-\overline{\lm_0})^{-1}P$, then $$\Lambda_{{\rm ess},\eps}\big((T_n)_{n\in\N}\big)=\emptyset.$$
If, in addition, $(T_n-\lm_0)^{-1}P_n\s (T-\lm_0)^{-1}P$, then $$\Lambda_{{\rm ess},\eps}\big((T_n^*)_{n\in\N}\big)^*=\emptyset.$$
\end{enumerate}
\end{prop}

\begin{proof}
i)
The proof is similar to the one of  Proposition~\ref{propdisccomp}~i).

ii)
Assume that there exists $\lm\in\Lambda_{{\rm ess},\eps}\big((T_n)_{n\in\N}\big)$. 
Then there are an infinite subset $I\subset\N$ and $x_n\in\dom(T_n)$, $n\in I$, with 
$$\|x_n\|=1, \quad x_n\stackrel{w}{\tolong} 0, \quad \|(T_n-\lm)x_n\|\tolong\eps, \quad n\in I, \quad n\to\infty.$$
Define $$y_n:=(T_n-\lm)x_n, \quad n\in I.$$
By  Lemma~\ref{lemmaClaim}, $y_n\stackrel{w}{\to}0$ as $n\in I$, $n\to\infty$.
Hence $(T_n-\lm_0)x_n=y_n+(\lm-\lm_0)x_n\stackrel{w}{\to}0$ and thus, by the assumptions and Lemma~\ref{lemmadisccompletelycont}~ii),
$$x_n=(T_n-\lm_0)^{-1}\big(y_n+(\lm-\lm_0)x_n)\tolong 0, \quad n\in I, \quad n\to\infty.$$
The obtained contradiction to $\|x_n\|=1$, $n\in I$, proves the first claim.

The second claim is obtained analogously, using that $((T_n^*-\overline{\lm_0})^{-1})_{n\in\N}$ is discretely compact  by~\cite[Proposition~2.10]{Boegli-chap1}.
\end{proof}

We prove a perturbation result for $\Lambda_{{\rm ess},\eps}\big((T_n)_{n\in\N}\big)$ and, in claim~ii), also for $\sigma_{\rm ess} \big((T_n)_{n\in\N}\big)$;
for the latter we use that the assumptions used here imply the assumptions of Theorem~\ref{thmreldisccomppert}.

\begin{theorem}\label{prop.pert.pseudo}
\begin{enumerate}
\item[\rm i)]
Let $B_n\in L(H_n),\,n\in\N$.
 If the sequence $(B_n)_{n\in\N}$ is discretely compact and $(B_n^*P_n)_{n\in\N}$ is strongly convergent, then 
$$\Lambda_{{\rm ess},\eps}\left((T_n+B_n)_{n\in\N}\right)=\Lambda_{{\rm ess},\eps}\left((T_n)_{n\in\N}\right).$$
If, in addition, $(B_nP_n)_{n\in\N}$ is strongly convergent, then 
$$\Lambda_{{\rm ess},\eps}\left(((T_n+B_n)^*)_{n\in\N}\right)^*=\Lambda_{{\rm ess},\eps}\left((T_n^*)_{n\in\N}\right)^*.$$

\item[\rm ii)]
Let $S$ and $S_n$, $n\in\N$, be linear operators in $H$ and $H_n$, $n\in\N$, with $\dom(T)\subset\dom(S)$ and
$\dom(T_n)\subset\dom(S_n)$, $n\in\N$, respectively.
Assume that there exist $\lm_0\in\underset{n\in\N}{\bigcap}\rho(T_n)\cap\rho(T)$ and $\gamma_{\lambda_0}<1$ such that
\begin{enumerate}
\item[\rm(a)]
$\|S(T-\lm_0)^{-1}\|<1$ and $\|S_n(T_n-\lm_0)^{-1}\|\leq\gamma_{\lambda_0}$ for all $n\in\N$;
\item[\rm(b)]
the sequence  $\big(S_n(T_n-\lm_0)^{-1}\big)_{n\in\N}$ is discretely compact;
\item[\rm(c)]
we have\vspace{-2mm}
$$\begin{array}{rl}(T_n^*-\overline{\lm_0})^{-1}P_n\hspace{-2mm}&\slong (T^*-\lbar{\lm_0})^{-1}P, \\
 (S_n(T_n-\lm_0)^{-1})^*P_n\hspace{-2mm}&\slong (S(T-{\lm_0})^{-1})^*P, \end{array}\quad n\to\infty.$$
\end{enumerate}
Then the sums $A:=T+S$ and $A_n:=T_n+S_n$, $n\in\N$, satisfy 
 $(A_n^*-\overline{\lm_0})^{-1}P_n\s (A^*-\lbar{\lm_0})^{-1}P$ and
\beq \label{eq.pert.pseudo}
\Lambda_{{\rm ess},\eps}\big((A_n)_{n\in\N}\big)=\Lambda_{{\rm ess},\eps}\big((T_n)_{n\in\N}\big), \quad \sigma_{\rm ess} \big((A_n)_{n\in\N}\big)=\sigma_{\rm ess} \big((T_n)_{n\in\N}\big).
\eeq

\item[\rm iii)]
Let $S$  be a linear operator in $H$ with $\dom(T)\subset\dom(S)$. If there exists $\lm_0\in\rho(T)$ such that 
$\|S(T-\lm_0)^{-1}\|<1$ and $S(T-\lm_0)^{-1}$ is compact, then
$$\Lambda_{{\rm ess},\eps}(T+S)=\Lambda_{{\rm ess},\eps}(T).$$ 
\end{enumerate}
\end{theorem}

\begin{proof} 
i)
The proof is analogous to the proof of Theorem~\ref{thmreldisccomppert}~i).

ii)
The proof relies on the following claim, which we prove at the end.

\noindent
\emph{Claim}: We have $\lm_0\in\underset{n\in\N}{\bigcap}\rho(A_n)\cap\rho(A)$,
the sequences 
\beq \label{eq.disccompSnAn}
\big((T_n-\lm_0)^{-1}-(A_n-\lm_0)^{-1}\big)_{n\in\N}, \quad \big(S_n(A_n-\lm_0)^{-1}\big)_{n\in\N}
\eeq
are discretely compact and, in the limit $n\to\infty$,
\beq \label{eq.strongconvSnAnadj}
\begin{array}{rl}(A_n^*-\overline{\lm_0})^{-1}P_n\hspace{-2mm}&\slong (A^*-\lbar{\lm_0})^{-1}P, \\[1mm]
 ((T_n-\lm_0)^{-1}-(A_n-\lm_0)^{-1})^*P_n\hspace{-2mm}&\slong ((T-\lm_0)^{-1}-(A-{\lm_0})^{-1})^*P,\\[1mm]
 (S_n(A_n-\lm_0)^{-1})^*P_n\hspace{-2mm}&\slong (S(A-{\lm_0})^{-1})^*P. \end{array}
\eeq

The second equality in~\eqref{eq.pert.pseudo} follows from the above Claim and Theorem~\ref{thmreldisccomppert}~ii). 

Now let $\lm\in  \Lambda_{{\rm ess},\eps}\big((A_n)_{n\in\N}\big)$.
Then there exist an infinite subset $I\subset\N$ and $x_n\in\dom(A_n)$, $n\in I$, with 
$$\|x_n\|=1, \quad x_n\stackrel{w}{\tolong} 0, \quad \|(A_n-\lm)x_n\|\tolong\eps, \quad n\in I, \quad n\to\infty.$$
Define $$y_n:=(A_n-\lm)x_n, \quad n\in I.$$
By Lemma~\ref{lemmaClaim}, we conclude $y_n\stackrel{w}{\to}0$ as $n\in I$, $n\to\infty$.
Then $(A_n-\lm_0)x_n=y_n+(\lm-\lm_0)x_n\stackrel{w}{\to}0$ and thus, by the above Claim and Lemma~\ref{lemmadisccompletelycont}~ii),
$$S_nx_n=S_n(A_n-\lm_0)^{-1}\big(y_n+(\lm-\lm_0)x_n)\tolong 0, \quad n\in I, \quad n\to\infty.$$
Therefore, $\|(T_n-\lm)x_n\|\leq \|(A_n-\lm)x_n\|+\|S_nx_n\|\to \eps$ and hence $\lm\in  \Lambda_{{\rm ess},\eps}\big((T_n)_{n\in\N}\big)$.

The reverse inclusion $\Lambda_{{\rm ess},\eps}\big((T_n)_{n\in\N}\big)\subset \Lambda_{{\rm ess},\eps}\big((A_n)_{n\in\N}\big)$ is proved analogously, using that $\big(S_n(T_n-\lm_0)^{-1}\big)_{n\in\N}$ is discretely compact and $(S_n(T_n-\lm_0)^{-1})^*P_n\s(S(T-\lm_0)^{-1})^*P$ by assumptions~(b),~(c).

\emph{Proof of Claim}:
A Neumann series argument implies that, for every $n\in\N$, we have $\lm_0\in\rho(A_n)$ and
\beq\label{eq.AnNeumann}
\begin{aligned}
(A_n-\lm_0)^{-1}&=(T_n-\lm_0)^{-1}(I+S_n(T_n-\lm_0)^{-1})^{-1},\\
(A_n^*-\overline{\lm_0})^{-1}&=\big(I+(S_n(T_n-\lm_0)^{-1})^*\big)^{-1}(T_n^*-\overline{\lm_0})^{-1},\\
(S_n(A_n-\lm_0)^{-1})^*&=\big(I+(S_n(T_n-\lm_0)^{-1})^*\big)^{-1}(S_n(T_n-\lm_0)^{-1})^*;
\end{aligned}
\eeq
for $A$, $S$, $T$ we obtain analogous equalities.
Now we apply~\cite[Lemma~3.2]{Boegli-chap1} to $B=(S(T-\lm_0)^{-1})^*$ and $B_n=(S_n(T_n-\lm_0)^{-1})^*$, $n\in\N$; note that $-1\in\Delta_b\big((B_n)_{n\in\N}\big)\cap\rho(B)$ by assumption~(a).
Hence we obtain 
$$\big(I+(S_n(T_n-\lm_0)^{-1})^*\big)^{-1}P_n\slong \big(I+(S(T-\lm_0)^{-1})^*\big)^{-1}P.$$
Now the strong convergences~\eqref{eq.strongconvSnAnadj} follow from~\eqref{eq.AnNeumann} and assumption~(c).
To prove discrete compactness of the sequences in~\eqref{eq.disccompSnAn}, we use that
\begin{align*}
(T_n-\lm_0)^{-1}-(A_n-\lm_0)^{-1}&=(A_n-\lm_0)^{-1}S_n(T_n-\lm_0)^{-1},\\
S_n(A_n-\lm_0)^{-1}&=S_n(T_n-\lm_0)^{-1}(I+S_n(T_n-\lm_0)^{-1})^{-1}.
\end{align*}
Now the claims are obtained by  \cite[Lemma~2.8~i),~ii)]{Boegli-chap1} and using assumptions~(a),~(b) and $(A_n^*-\overline{\lm_0})^{-1}P_n\s (A^*-\lbar{\lm_0})^{-1}P$ by~\eqref{eq.strongconvSnAnadj}.

iii) The assertion follows from claim~ii) applied to $T_n=T$, $S_n=S$, $n\in\N$.
\end{proof}

\subsection{Proofs of pseudospectral convergence results}\label{subsectionproofofmainresults.pseudo}

First we prove the $\eps$-pseudospectral inclusion result.

\begin{proof}[Proof of Theorem~{\rm\ref{thm.inclusion.prelim}}]
i)
The assumption $T_n\gsr T$ and Lemma~\ref{AndAJndJnew} imply that there are $x_n\in\dom(T_n)$, $n\in\N$, with
$\|x_n\|=1$, $\|x_n-x\|\to 0$, $\|T_nx_n-Tx\|\to 0$.
Hence there exists $n_{\lm}\in\N$ such that, for all $ n\geq n_{\lm}$,
$$\|(T_n-\lm)x_n\|\leq \|(T-\lm)x\|+\|T_nx_n-Tx\|+|\lm|\|x_n-x\|<\eps.$$
Therefore, $\lm\in\sigma_{{\rm app},\eps}(T_n)\subset\sigma_{\eps}(T_n)$ for all $n\geq n_{\lm}$.

ii)
By Lemma~\ref{lemma.eps.T}~ii), $\sigma_{\eps}(T)=\sigma_{{\rm app},\eps}(T)\cup \sigma_{{\rm app},\eps}(T^*)^*$.
Now the assertion follows from claim~i) and the assumptions  $T_n\gsr T$,  $T_n^*\gsr T^*$.
\end{proof}

Now we confine the set of pseudospectral pollution.

\begin{prop}\label{thmexactness2.prelim}
Suppose that $T_n\gsr T$ and $T_n^*\gsr T^*$.
Let $\lm\in\rho(T)\cap\Delta_b\big((T_n)_{n\in\N}\big)$ and let $\lm_n\in\C$, $n\in\N$, satisfy $\lm_n\to\lm$, $n\to\infty$.
Then
$$M:=\limsup_{n\to\infty}\|(T_n-\lm_n)^{-1}\|=\limsup_{n\to\infty}\|(T_n-\lm)^{-1}\|\geq \|(T-\lm)^{-1}\|;$$
if the inequality is strict, then
$$\lm\in \Lambda_{{\rm ess},\frac{1}{M}}\big((T_n)_{n\in\N}\big)\cap  \Lambda_{{\rm ess},\frac{1}{M}}\big((T_n^*)_{n\in\N}\big)^*.$$
\end{prop}

\begin{proof}
First we prove that 
\beq\label{eq.M}
M=\limsup_{n\to\infty}\|(T_n-\lm_n)^{-1}\|=\limsup_{n\to\infty}\|(T_n-\lm)^{-1}\|.
\eeq
Since  $\lm\in\Delta_b\big((T_n)_{n\in\N}\big)$, we have $C:=\sup_{n\in\N}\|(T_n-\lm)^{-1}\|<\infty$.
A Neumann series argument yields that, for all $n\in \N$ so large that $|\lm_n-\lm|<1/C$,
$$\|(T_n-\lm_n)^{-1}\|=\big\|(T_n-\lm)^{-1}(I-(\lm_n-\lm)(T_n-\lm)^{-1})^{-1}\big\|\leq \frac{C}{1-|\lm_n-\lm|C}.$$
The first resolvent identity implies
\begin{align*}
\big|\|(T_n-\lm)^{-1}\|-\|(T_n-\lm_n)^{-1}\|\big|
&\leq |\lm-\lm_n|\|(T_n-\lm_n)^{-1}\|\|(T_n-\lm)^{-1}\|\\
&\leq  |\lm-\lm_n|\frac{C^2}{1-|\lm_n-\lm| C}.
\end{align*}
The right hand side converges to $0$ since $\lm_n\to\lm$. This proves~\eqref{eq.M}.

The inequality $$\limsup_{n\to\infty}\|(T_n-\lm)^{-1}\|\geq \|(T-\lm)^{-1}\|$$
follows from Theorem~\ref{thm.inclusion.prelim}~ii).

Now assume that $M>\|(T-\lm)^{-1}\|$.
First note that $ (T_n^*-\overline{\lm})^{-1}(T_n-\lm)^{-1}$ is selfadjoint and
$$\|(T_n-\lm)^{-1}\|^2=\sup_{\|y\|=1}\langle (T_n^*-\overline{\lm})^{-1}(T_n-\lm)^{-1}y,y\rangle=\max\,\sigma_{\rm app}\big( (T_n^*-\overline{\lm})^{-1}(T_n-\lm)^{-1}\big).$$
Therefore, $$M\in\sigma_{\rm app}\big(\big((T_n^*-\overline{\lm})^{-1}(T_n-\lm)^{-1}\big)_{n\in\N}\big).$$
By the assumptions $T_n\gsr T$, $T_n^*\gsr T^*$ and $\lm\in\Delta_b\big((T_n)_{n\in\N}\big)\cap\rho(T)$, we obtain, using~\cite[Proposition~2.16]{Boegli-chap1},
$$ (T_n^*-\overline{\lm})^{-1}(T_n-\lm)^{-1}P_n\slong  (T^*-\overline{\lm})^{-1}(T-\lm)^{-1}P, \quad n\to\infty.$$ 
Moreover, \cite[Lemma~3.2]{Boegli-chap1} yields $(T_n^*-\overline{\lm})^{-1}(T_n-\lm)^{-1}\gsr  (T^*-\overline{\lm})^{-1}(T-\lm)^{-1}$.
Now Proposition~\ref{lemmasigmaappofTn}~ii) implies that 
$$M^2\in\sigma_p\big((T^*-\overline{\lm})^{-1}(T-\lm)^{-1}\big)\cup\sigma_{\rm ess} \big(\big((T_n^*-\overline{\lm})^{-1}(T_n-\lm)^{-1}\big)_{n\in\N}\big).$$

\emph{First case}: If $M^2\in\sigma_p\big((T^*-\overline{\lm})^{-1}(T-\lm)^{-1}\big)$, then there exists $y\in H$ with $\|y\|=1$ such that 
$$0=\big\langle \big((T^*-\overline{\lm})^{-1}(T-\lm)^{-1}-M^2\big)y,y\big\rangle=\|(T-\lm)^{-1}y\|^2-M^2.$$
So we arrive at the contradiction $\|(T-\lm)^{-1}\|\geq M$.

\emph{Second case}:
If $M^2\in\sigma_{\rm ess} \big(\big((T_n^*-\overline{\lm})^{-1}(T_n-\lm)^{-1}\big)_{n\in\N}\big)$,
then  the mapping result in Theorem~\ref{mappingthm}
implies that $1/M^2\in \sigma_{\rm ess} \big(((T_n-\lm)(T_n^*-\overline{\lm}))_{n\in\N}\big).$
Hence there exist an infinite subset $I\subset\N$ and $x_n\in\dom((T_n-\lm)(T_n^*-\lbar{\lm}))\subset \dom(T_n^*)$, $n\in I$, with $\|x_n\|=1$, $x_n\stackrel{w}{\to}0$ and
$\|((T_n-\lm)(T_n^*-\overline{\lm})-1/M^2)x_n\|\to 0$.
So we arrive at
$$\|(T_n^*-\overline{\lm})x_n\|^2=\langle (T_n-\lm)(T_n^*-\overline{\lm})x_n,x_n\rangle\tolong \frac{1}{M^2}, \quad n\in I, \quad n\to\infty.$$
This implies $\lm\in  \Lambda_{{\rm ess},1/M}\big((T_n^*)_{n\in\N}\big)^*$.

Since $\|(T-\lm)^{-1}\|=\|(T^*-\overline{\lm})^{-1}\|$ and  $\|(T_n-\lm)^{-1}\|=\|(T_n^*-\overline{\lm})^{-1}\|$, we obtain analogously that $\lm\in  \Lambda_{{\rm ess},1/M}\big((T_n)_{n\in\N}\big)$.
\end{proof}

Next we prove the $\eps$-pseudospectral exactness result.

\begin{proof}[Proof of Theorem~{\rm\ref{thmexactness2}}]
i) Let $\lm\in\overline{\sigma_{\eps}(T)}$. Assume that the claim is false, i.e.\ 
\beq\label{eq.alpha}
\alpha:=\limsup_{n\to\infty}{\rm dist}(\lm,\overline{\sigma_{\eps}(T_n)})>0.
\eeq
Choose $\widetilde\lm\in\sigma_{\eps}(T)$ with $|\lm-\widetilde\lm|<\alpha/2$. 
By Theorem~\ref{thm.inclusion.prelim}~ii),
there exists $n_{\widetilde\lm}\in\N$ such that $\widetilde\lm\in\sigma_{\eps}(T_n)\subset\overline{\sigma_{\eps}(T_n)}$, $n\geq n_{\widetilde\lm}$, which is a contradiction to~\eqref{eq.alpha}. 

ii) Choose $\lm\in\C\backslash\overline{\sigma_{\eps}(T)}$ outside the set in~\eqref{eq.defKgeneral} or~\eqref{eq.defK2}, respectively. Assume that it is a point of $\eps$-pseudospectral pollution, i.e.\  there exist an infinite subset $I\subset\N$  and $\lm_n\in\overline{\sigma_{\eps}(T_n)}$, $n\in I$, with $\lm_n\to\lm$. 
By the choice of $\lm$ and  Proposition~\ref{thmregionofbddvsseqessspectrum}~i), we arrive at  $\lm\in\rho(T)\cap\Delta_b\big((T_n)_{n\in\N}\big)$. 

Since $\lm_n\in\overline{\sigma_{\eps}(T_n)}$, $n\in I$, we have $$M:=\limsup_{n\in I\atop n\to\infty}\|(T_n-\lm_n)^{-1}\|\geq\frac{1}{\eps}.$$
Now, by Proposition~\ref{thmexactness2.prelim} and using $\lm\notin\Lambda_{{\rm ess},(0,\eps]}$, we conclude that $\|(T-\lm)^{-1}\|=1/{\eps}$. By Theorem~\ref{thmconstresnorm}~ii), the level set $\{\lm\in\rho(T):\,\|(T-\lm)^{-1}\|=1/\eps\}$ does not have an open subset. Hence we arrive at the contradiction $\lm\in\overline{\sigma_{\eps}(T)}$, which proves the claim.

iii) 
The proof is similar to the one of Theorem~\ref{mainthmspectralexactness}~iii).
Assume that the claim is false. Then there exist $\alpha>0$, an infinite subset $I\subset\N$ and $\lm_n\in K$, $n\in I$, such that one of the following holds:
\begin{enumerate}
\item[\rm(1)]  $\lm_n\in\overline{\sigma_{\eps}(T_n)}$ and ${\rm dist}(\lm_n,\overline{\sigma_{\eps}(T)}\cap K)>\alpha$ for every $n\in I$;
\item[\rm(2)]  $\lm_n\in\overline{\sigma_{\eps}(T)}$ and ${\rm dist}(\lm_n,\overline{\sigma_{\eps}(T_n)}\cap K)>\alpha$ for every $n\in I$.
\end{enumerate}
Note that, in both cases (1) and (2), the compactness of $K$ implies that there exist $\lm\in K$ and an infinite subset $J\subset I$ such that $(\lm_n)_{n\in J}$ converges to $\lm$.

First we consider case (1).
Claim~ii) and the assumptions on $K$ imply that $\lm\in\overline{\sigma_{\eps}(T)}\cap K$ and hence 
$|\lm_n-\lm|\geq {\rm dist}(\lm_n,\overline{\sigma_{\eps}(T)}\cap K)>\alpha$, $n\in J$, a contradiction to $\lm_n\to\lm$.

Now assume that (2) holds. The assumption $\overline{\sigma_{\eps}(T)}\cap K=\overline{\sigma_{\eps}(T)\cap K}$
implies that $(\lm_n)_{n\in J}\subset \overline{\sigma_{\eps}(T)\cap K}$ and thus $\lm\in\overline{\sigma_{\eps}(T)\cap K}$.
Choose $\widetilde\lm\in\sigma_{\eps}(T)\cap K$ with $|\lm-\widetilde\lm|<\alpha/2$. 
By Theorem~\ref{thm.inclusion.prelim}~ii), there exists $n_{\widetilde\lm}\in\N$ such that $\widetilde\lm\in\sigma_{\eps}(T_n)\cap K$ for every $n\geq n_{\widetilde \lm}$. Therefore 
$|\lm_n-\widetilde\lm|\geq {\rm dist}(\lm_n,\overline{\sigma_{\eps}(T_n)}\cap K)>\alpha$ for every $n\in J$ with $n\geq n_{\widetilde\lm}$. Since $|\lm_n-\widetilde\lm|\to |\lm-\widetilde\lm|<\alpha/2$, we arrive at a contradiction.
This proves the claim.
\end{proof}

Finally we prove the result about operators that have constant resolvent norm on an open set.

\begin{proof}[Proof of Theorem{\rm~\ref{thmconstresnorm}}]
i)
Let $\lm_0\in U$.
By proceeding as in the proof of Theorem~\ref{thmexactness2} (with $T_n=T$, $\lm_n=\lm=\lm_0$, $n\in\N$, and $M=1/\eps$), we obtain $$\frac{1}{\eps^2}\in\sigma_p((T^*-\overline{\lm_0})^{-1}(T-\lm_0)^{-1})\cup\sigma_{\rm ess}((T^*-\overline{\lm_0})^{-1}(T-\lm_0)^{-1}),$$
and the second case $1/\eps^2\in \sigma_{\rm ess}((T^*-\overline{\lm_0})^{-1}(T-\lm_0)^{-1})$ implies $\lm_0\in\Lambda_{{\rm ess},\eps}(T^*)^*$.
If however $1/\eps^2\in \sigma_p((T^*-\overline{\lm_0})^{-1}(T-\lm_0)^{-1})$, then there exists $y\in H$ with $\|y\|=1$ such that $\|(T-\lm)^{-1}y\|=1/\eps$. 
Hence $x:=(T-\lm)^{-1}y\neq 0$ satisfies $\|(T-\lm)x\|/\|x\|=\eps$.
Note that $\mu\mapsto\|(T-\mu)x\|$ is a non-constant subharmonic function on $\C$ and thus satisfies the maximum principle. Therefore, in every open neighbourhood of $\lm_0$ there exist points $\mu$ such that $\|(T-\mu)x\|/\|x\|<\eps$ and thus  $\lm_0\in\overline{\sigma_{\eps}(T)}$, a contradiction.

Since $\|(T-\lm)^{-1}\|=\|(T^*-\overline\lm)^{-1}\|=1/\eps$ for every $\lm\in U$, we analogously obtain $\lm_0\in\Lambda_{{\rm ess},\eps}(T)$.
So there exists a sequence $(x_n)_{n\in\N}\subset\dom(T)$ with $\|x_n\|=1$, $x_n\stackrel{w}{\to}0$ and $\|(T-\lm_0)x_n\|\to\eps$.
Define $$e_n:=\frac{(T-\lm_0)x_n}{\|(T-\lm_0)x_n\|}, \quad n\in\N.$$
 Then $\|e_n\|=1$ and  $e_n\stackrel{w}{\to}0$ by Lemma~\ref{lemmaClaim} applied to $T_n=T$.
In addition, $\|(T-\lm_0)^{-1}e_n\|\to \|(T-\lm_0)^{-1}\|= 1/\eps$.
Analogously as in the proof of~\cite[Theorem~3.2]{Boegli-2014-80}, using the old results~ \cite[Lemmas~1.1,3.0]{Globevnik-1974-15} by Globevnik and Vidav, one may show that 
\beq \label{eq.res.squared}
(T-\lm_0)^{-2}e_n\tolong 0, \quad n\to\infty;
\eeq
note that ~\cite[Theorem~3.2]{Boegli-2014-80} was proved for complex uniformly convex Banach spaces and is thus, in particular, valid for Hilbert spaces (see~\cite{Globevnik-1975-47} for a discussion about complex uniform convexity).

Let $\lm\in\C\backslash\underset{K \text{ compact}}{\bigcap}\sigma(T+K).$ 
Then there exists a compact operator $K\in L(H)$ such that $\lm\in\rho(T+K)$.
The second resolvent identity applied twice yields
\begin{align*}
&(T+K-\lm)^{-1}-(T-\lm_0)^{-1}\\
&=(T+K-\lm)^{-1}(-K+\lm-\lm_0)(T-\lm_0)^{-1}\\
&=-\big(I+(T+K-\lm)^{-1}(-K+\lm-\lm_0)\big)(T-\lm_0)^{-1}K(T-\lm_0)^{-1}\\
&\quad+(\lm-\lm_0)\big(I+(T+K-\lm)^{-1}(-K+\lm-\lm_0)\big)(T-\lm_0)^{-2}.
\end{align*}
Since $\widetilde K:=-\big(I+(T+K-\lm)^{-1}(-K+\lm-\lm_0)\big)(T-\lm_0)^{-1}K(T-\lm_0)^{-1}$ is compact and hence completely continuous,
the weak convergence $e_n\stackrel{w}{\to}0$ yields $\widetilde K e_n\to 0$.
Using~\eqref{eq.res.squared} in addition, we conclude  $\big((T+K-\lm)^{-1}-(T-\lm_0)^{-1}\big)e_n\to 0$ and hence 
$$\lim_{n\to\infty}\|(T+K-\lm)^{-1}e_n\|=\lim_{n\to\infty}\|(T-\lm_0)^{-1}e_n\|=\frac{1}{\eps}.$$
Now define 
$$w_n:=\frac{(T+K-\lm)^{-1}e_n}{\|(T+K-\lm)^{-1}e_n\|}, \quad n\in\N.$$
Then $\|w_n\|=1$,  $w_n\stackrel{w}{\to}0$ and $\|(T+K-\lm)w_n\|=\|(T+K-\lm)^{-1}e_n\|^{-1}\to \eps$.
Therefore, $\lm\in\Lambda_{{\rm ess},\eps}(T+K)$.
Theorem~\ref{prop.pert.pseudo}~i) applied to $T_n=T$ and $B_n=K$ yields $\Lambda_{{\rm ess},\eps}(T+K)=\Lambda_{{\rm ess},\eps}(T)$. So arrive at $\lm\in\Lambda_{{\rm ess},\eps}(T)$.

In an analogous way as for~\eqref{eq.res.squared}, one may show that there exists a normalised sequence $(f_n)_{n\in\N}\subset H$ with $f_n\stackrel{w}{\to}0$ and $(T^*-\overline{\lm_0})^{-2}f_n\to 0$. 
So, by proceeding as above, we obtain $\lm\in\Lambda_{{\rm ess},\eps}(T^*)^*$.

ii) The claim follows from claim~i) and Proposition~\ref{proplimsigmaess}~i).
\end{proof}

\section{Applications and Examples}\label{sectionapplications}
In this section we discuss applications  to the Galerkin method for infinite matrices (Subsection~\ref{subsectionmatrices}) and to the domain truncation method for differential operators (Subsection~\ref{subsectionPDE}).

\subsection{Galerkin approximation of block-diagonally dominant matrices}\label{subsectionmatrices}
In this subsection we consider an operator $A$ in $l^2(\K)$ (where $\K=\N$ or $\K=\Z$) whose matrix representation (identified with $A$) with respect to the standard orthonormal basis $\{e_j:\,j\in\K\}$
can be split as $A=T+S$. 
Here $T$ is block-diagonal, i.e.\ there exist  $m_k\in\N$ with  
$$T={\rm diag}(T_k:\,k\in\K), \quad T_k\in\C^{m_k\times m_k}.$$
We further assume that $\dom(T)\subset\dom(S)$, $\dom(T^*)\subset\dom(S^*)$ and that there exists $\lm_0\in\rho(T)$ such that $S(T-\lm_0)^{-1}$ is compact and
\beq\label{eq.relbd.blockdiag}\|S(T-\lm_0)^{-1}\|<1, \quad \|S^*(T^*-\overline{\lm_0})^{-1}\|<1.\eeq
Define, for $n\in\N$, 
$$j_n:=\begin{cases}-\sum\limits_{k=-n}^{0} m_k, &\K=\Z,\\ 1, &\K=\N,\end{cases},\qquad J_n:=\sum_{k=1}^n m_k.$$ 
Let $P_n$ be the orthogonal projection of $l^2(\K)$ onto $H_n:={\rm span}\{e_j:\,j_n\leq j\leq J_n\}$.
It is easy to see that $P_n\s I$.

\begin{theorem}\label{thmblockdiag}
Define $A_n:=P_nA|_{H_n}$, $n\in\N$. 
\begin{enumerate}
\item[\rm i)] We have  $A_n\gsr A$ and $A_n^*\gsr A^*$.

\item[\rm ii)] The limiting essential spectra satisfy 
\begin{align*}
&\sigma_{\rm ess}\big((A_n)_{n\in\N}\big)\cup \sigma_{\rm ess}\big((A_n^*)_{n\in\N}\big)^*\\
&=\{\lm\in\C:\,\exists\, I\subset\N \,\text{with }\|(T_n-\lm)^{-1}\|\to\infty,\,n\in I, \,n\to\infty\}\\
&=\sigma_{\rm ess}(A); 
\end{align*}
hence no spectral pollution occurs for the approximation $(A_n)_{n\in\N}$ of~$A$, and for every isolated $\lm\in\sigma_{\rm dis}(A)$ there exists a sequence of $\lm_n\in\sigma(A_n)$, $n\in\N$, with $\lm_n\to \lm$ as $n\to\infty$.
\item[\rm iii)] The limiting essential $\eps$-near spectrum satisfies
\begin{align*}
&\Lambda_{{\rm ess},\eps}\big((A_n)_{n\in\N}\big)\\
&=\bigg\{\lm\in\C:\,\exists\, I\subset\N \,\text{with }\|(T_n-\lm)^{-1}\|\to\frac{1}{\eps},\,n\in I, \,n\to\infty\bigg\}\\
&=\Lambda_{{\rm ess},\eps}(A);
\end{align*}
hence if $A$ does not have constant resolvent norm {\rm(}$=1/\eps${\rm)} on an open set, then $(A_n)_{n\in\N}$ is an $\eps$-pseudospectrally exact approximation of $A$.
\end{enumerate}
\end{theorem}

\begin{proof}
First note that the adjoint operators satisfy $A_n^*=T_n^*+S_n^*$, and since, by~\eqref{eq.relbd.blockdiag}, $S$ is $T$-bounded and $S^*$ is $T^*$-bounded with  relative bounds $<1$, \cite[Corollary~1]{Hess-Kato} implies $A^*=T^*+S^*$.
In addition, for any $n\in\N$,
\beq\label{eq.prodconv}
\begin{aligned}
(T_n-\lm_0)^{-1}P_n&=P_n(T-\lm_0)^{-1}, \\
 (T_n^*-\overline{\lm_0})^{-1}P_n&=P_n(T^*-\overline{\lm_0})^{-1},\\
S_n(T_n-\lm_0)^{-1}P_n&=P_nS(T-\lm_0)^{-1}, \\
 S_n^*(T_n^*-\overline{\lm_0})^{-1}P_n&=P_nS^*(T^*-\overline{\lm_0})^{-1},\\
(S_n(T_n-\lm_0)^{-1})^*P_n|_{\dom(S^*)}&=P_n(T^*-\overline{\lm_0})^{-1}S^*=P_n(S(T-\lm_0)^{-1})^*|_{\dom(S^*)}. 
\end{aligned}
\eeq
Now, using~\eqref{eq.prodconv} everywhere, we check that the assumptions of Theorem~\ref{prop.pert.pseudo}~ii),~iii) are satisfied. 
\begin{enumerate}
\item[\rm(a)]
We readily conclude
\beq\label{eq.prodest}
\|S_n(T_n-\lm_0)^{-1}\|\leq \|S(T-\lm_0)^{-1}\|<1.
\eeq

\item[\rm(b)] 
The sequence of operators $S_n(T_n-\lm_0)^{-1}=P_nS(T-\lm_0)^{-1}|_{H_n}$, $n\in\N$, is discretely compact since $S(T-\lm_0)^{-1}$ is compact and $P_n\s I$. 

\item[\rm(c)]
The strong convergence  $(T_n^*-\overline{\lm_0})^{-1}P_n\s (T^*-\overline{\lm_0})^{-1}$ follows immediately from $P_n\s I$.
In addition, since $\dom(S^*)$ is a dense subset, using~\eqref{eq.prodest} we obtain $(S_n(T_n-\lm_0)^{-1})^*P_n\s (S(T-\lm_0)^{-1})^*$. 
\end{enumerate}
Now Theorem~\ref{prop.pert.pseudo}~ii),~iii) implies $A_n^*\gsr A^*$ and
\begin{align*}
\sigma_{\rm ess}\big((A_n)_{n\in\N}\big)&=\sigma_{\rm ess}\big((T_n)_{n\in\N}\big),\quad  
\Lambda_{{\rm ess}, \eps}\big((A_n)_{n\in\N}\big)=\Lambda_{{\rm ess}, \eps}\big((T_n)_{n\in\N}\big),\\
 \Lambda_{{\rm ess}, \eps}(A)&=\Lambda_{{\rm ess},\eps}(T).
\end{align*}
In addition, since $S(T-\lm_0)^{-1}$ is assumed to be compact, \cite[Theorem IX.2.1]{edmundsevans} yields $\sigma_{\rm ess}(A)=\sigma_{\rm ess}(T)$.

In claim~i) it is left to be shown that $A_n\gsr A$. To this end, we use that~\eqref{eq.prodconv} and $P_n\s I$ imply
$$(T_n-\lm_0)^{-1}P_n\s (T-\lm_0)^{-1}, \quad S_n(T_n-\lm_0)^{-1}P_n\s S(T-\lm_0)^{-1}.$$
Now the claim follows from~\eqref{eq.prodest} and the perturbation result~\cite[Theorem~3.3]{Boegli-chap1}.

Note that  Theorem~\ref{thmreldisccomppert}~ii) implies $\sigma_{\rm ess}\big((A_n^*)_{n\in\N}\big)^*=\sigma_{\rm ess}\big((T_n^*)_{n\in\N}\big)^*$.
The identities in claim~ii) are obtained from 
\begin{align*}
&\sigma_{\rm ess}\big((T_n)_{n\in\N}\big)\cup \sigma_{\rm ess}\big((T_n^*)_{n\in\N}\big)^*\\
&=\{\lm\in\C:\,\exists\, I\subset\N \,\text{with }\|(T_n-\lm)^{-1}\|\to\infty,\,n\in I, \,n\to\infty\}
=\sigma_{\rm ess}(T). 
\end{align*}
Now the local spectral exactness follows from Theorem~\ref{mainthmspectralexactness}.

The assertion in~iii) follows from an analogous reasoning, using Theorem~\ref{thmexactness2};
note that if $T$ does not have constant ($=1/\eps$) resolvent norm on an open set, then $\Lambda_{{\rm ess},\eps}(T)\subset\partial\sigma_{\eps}(T)\subset\overline{\sigma_{\eps}(T)}$ and hence no $\eps$-pseudospectral pollution occurs. 
\end{proof}

\begin{example}
For points $b,d\in\C$ and sequences $(a_j)_{j\in\N},(b_j)_{j\in\N}$, $(c_j)_{j\in\N}$, $(d_j)_{j\in\N}\subset\C$ with
$$|a_j|\tolong\infty,\quad b_j\tolong b, \quad c_j\tolong 0, \quad d_j\tolong d, \quad j\to \infty,$$
define an unbounded operator $A$ in $l^2(\N)$ by
 $$A:=\bmat a_1 & b_1 &&&\\[1mm] c_1 & d_1 & b_2 &&\\[1mm]  & c_2 & a_2 & b_3 & \\[-0.4mm] && c_3 & d_2 & \ddots  \\ &&& \ddots & \ddots   \emat, \quad \dom(A):=\bigg\{(x_j)_{j\in\N}\in l^2(\N):\,\sum_{j\in\N}|a_jx_{2j-1}|^2<\infty\bigg\}.$$
For $n\in\N$ let $P_n$ be the orthogonal projection of $l^2(\N)$ onto the first $2n$ basis vectors, and define $A_n:=P_nA|_{\ran(P_n)}$.
Using Theorem~\ref{thmblockdiag}, we show that $$\sigma_{\rm ess}(A)=\{d\}, \quad \Lambda_{{\rm ess},\eps}(A)=\{\lm\in\C:\,|\lm-d|= \eps\},$$
that every $\lm\in\sigma_{\rm dis}(A)$ is an accumulation point of $\sigma(A_n)$, $n\in\N$, that no spectral pollution occurs and that $(A_n)_{n\in\N}$ is $\eps$-pseudospectrally exact.

To this end, define $$T:={\rm diag}(T_k:\,k\in\N), \quad T_k:=\bmat a_k & b_{2k-1}\\ 0 & d_k\emat, \quad \dom(T):=\dom(A).$$
Then it is easy to check that $S:=A-T$ is $T$-compact and the estimates in~\eqref{eq.relbd.blockdiag} are satisfied for all $\lm_0\in\C$ that are sufficiently far from $\sigma(T)$.
The essential spectrum $\sigma_{\rm ess}(T)$ consists of all accumulation points of $\sigma(T_k)=\{a_k,d_k\}$, $k\in\N$, i.e.\ $\sigma_{\rm ess}(T)=\{d\}$.
To find $\Lambda_{{\rm ess},\eps}(T)$, note that, in the limit $k\to\infty$,
$$\|(T_k-\lm)^{-1}\|=\left\|\bmat (a_k-\lm)^{-1} & -b_{2k-1}(a_k-\lm)^{-1}(d_k-\lm)^{-1}\\ 0 & (d_k-\lm)^{-1}\emat\right\|\tolong \frac{1}{|d-\lm|}.$$ 
This proves $\Lambda_{{\rm ess},\eps}(T)=\{\lm\in\C:\,|\lm-d| = \eps\}$. 
Now the claims follow from  Theorem~\ref{thmblockdiag} and since $\Lambda_{{\rm ess},\eps}(T)$ does not contain an open subset.
\end{example}

The following example is influenced by Shargorodsky's example~\cite[Theorems~3.2, 3.3]{Shargorodsky-2008-40} of an operator with constant resolvent norm on an open set and whose matrix representation is block-diagonal.
Here we perturb a block-diagonal operator with constant resolvent norm on an open set and arrive at an operator whose resolvent norm is also constant on an open set.

\begin{example}
Consider the \emph{neutral delay differential expression} $\tau$ defined by
$$(\tau f)(t):=\e^{\I t}(f''(t)+f''(t+\pi))+\e^{-\I t}f(t).$$
For an extensive treatment of neutral differential equations with delay, see the monograph~\cite{delay} (in particular Chapter~3 for second order equations).
Let $A$ be the realisation of $\tau$ in $L^2(-\pi,\pi)$ with domain 
$$\dom(A):=\bigg\{f\in L^2(-\pi,\pi):\,\begin{array}{l}f,f'\in {\rm AC_{loc}}(-\pi,\pi),\,\tau f\in L^2(-\pi,\pi),\\
f(-\pi)=f(\pi),\,f'(-\pi)=f'(\pi)\end{array}\bigg\},$$
where $f$ is continued $2\pi$-periodically.
With respect to the orthonormal basis $\{e_k:\,k\in\Z\}\subset \dom(A)$ with $e_k(t):=\e^{\I k t}/\sqrt{2\pi}$, the operator $A$ has the matrix representation 
\begin{align*}
A&=\bmat \ddots & \ddots &&&\\[1mm]  & T_{-1} & S_{-1} && \\[1mm] && T_0 & S_0 & \\[1mm] &&& T_1 & \ddots  \\[1mm] &&&& \ddots  \emat, \quad  
T_j=\bmat 0 & 1 \\ a_j & 0 \emat, \quad S_j=\bmat 0 & 0 \\ 1 & 0\emat,\\
 a_j&=8j^2, \quad \dom(A)=\bigg\{((u_j,v_j)^t)_{j\in\Z}:(u_j)_{j\in\Z},(v_j)_{j\in \Z}\in l^2(\Z),\,\sum_{j\in\Z}|a_{j}u_j|^2<\infty\bigg\}.
\end{align*}
We split $A$ to $T:={\rm diag}(T_j:\,j\in\Z)$, $\dom(T):=\dom(A)$, and $S:=A-T$. Note that $S$ on $\dom(S)=l^2(\Z)$ is bounded and $T$-compact and $S^*$ is $T^*$-compact,
but $T$ does not have compact resolvent. 
Next we prove the existence of $\lm_0\in\rho(T)$ such that the estimates in~\eqref{eq.relbd.blockdiag} are satisfied. 
To this end, let $\lm_0\in\I\R\backslash\{0\}$ and estimate
$$\|S(T-\lm_0)^{-1}\|=\sup_{j\in\Z}\|S_j(T_{j+1}-\lm_0)^{-1}\|=\sup_{j\in\Z}\left\|\frac{1}{a_j-\lm_0^2}\bmat 0 & 0\\ \lm_0 & 1 \emat\right\|\leq \frac{1+|\lm_0|}{|\lm_0|^2};$$
one may check that also $\|S^*(T^*-\overline{\lm_0})^{-1}\|\leq (1+|\lm_0|)/|\lm_0|^2$. Hence~\eqref{eq.relbd.blockdiag} is satisfied if $|\lm_0|$ is sufficiently large.

Let $P_n$ denote the orthogonal projection onto $$
H_n:={\rm span}\{e_k:\,k=-(2n),\dots, 2n-1\},$$
and let $A_n$ and $T_n$ denote the respective Galerkin approximations,
$$A_n:=P_nA|_{H_n}, \quad T_n:=P_nT|_{H_n}, \quad n\in\N.$$
Note that $\det(A_n-\lm)=\det(T_n-\lm)$ for every $\lm\in\C$, which implies $\sigma(A_n)=\sigma(T_n)=\big\{ \pm\sqrt{a_j}:\,j=-n,\dots,n\big\}$.
Hence Theorem~\ref{thmblockdiag}~ii) proves $$\sigma(A)=\big\{ \pm\sqrt{a_j}:\,j\in\Z\big\}=\{\pm \sqrt{8}\, j:\,j\in\N_0\}.$$

Now we study the pseudospectra of $A$ and $T$.
\begin{figure}[b]
\begin{center}
 \includegraphics[width=0.5\textwidth]{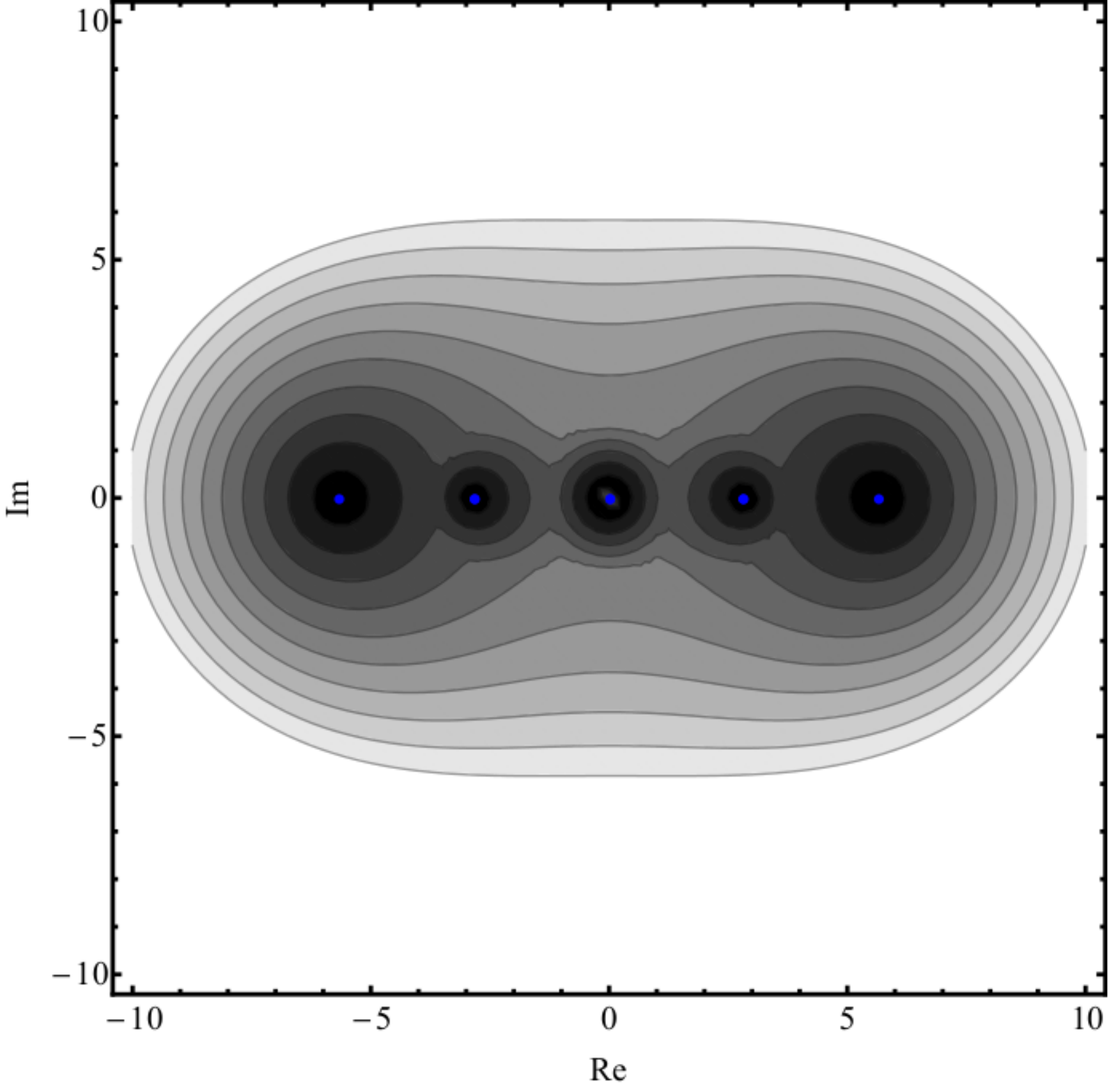}\\
 \includegraphics[width=0.5\textwidth]{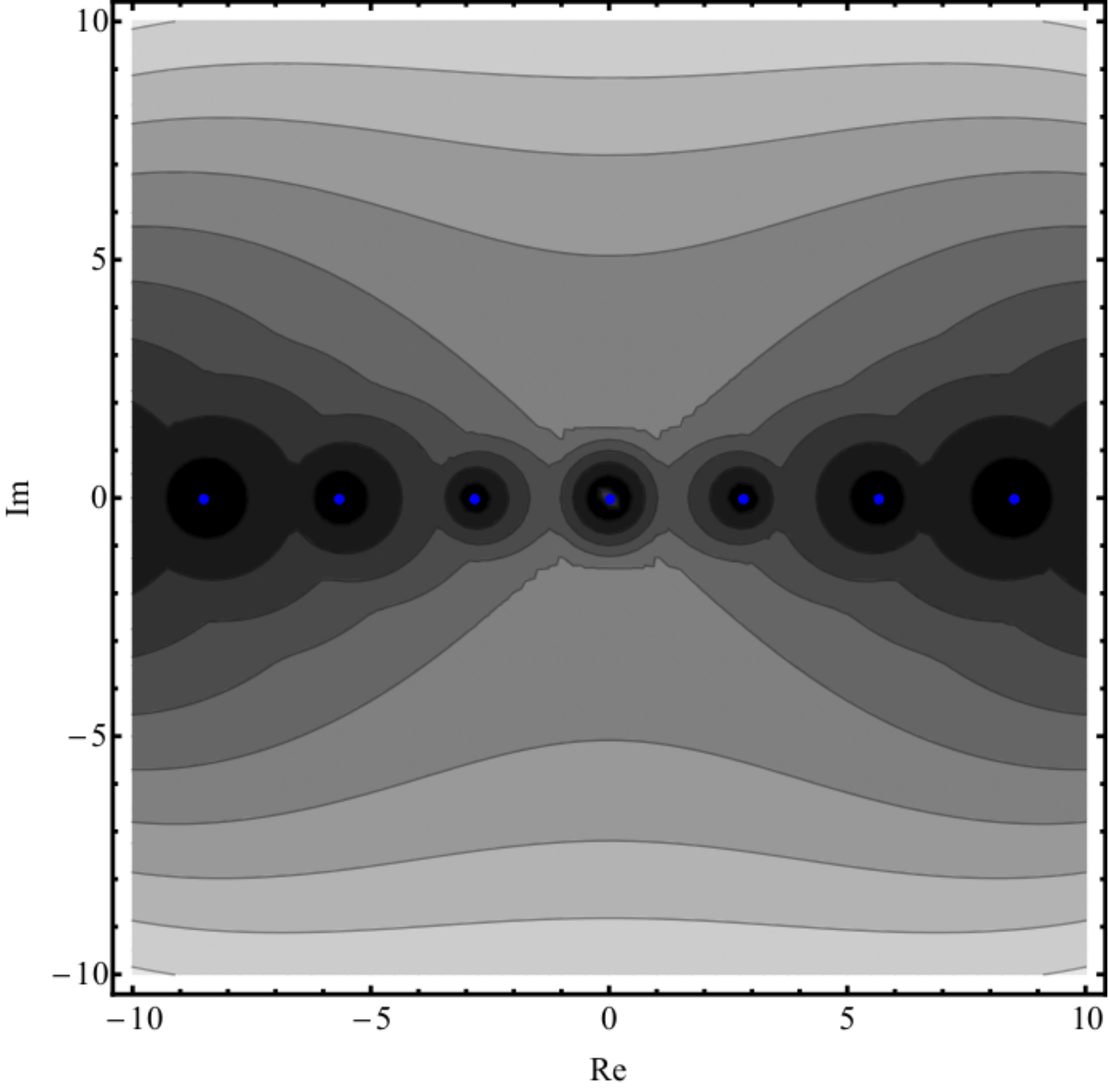}\\
 \includegraphics[width=0.5\textwidth]{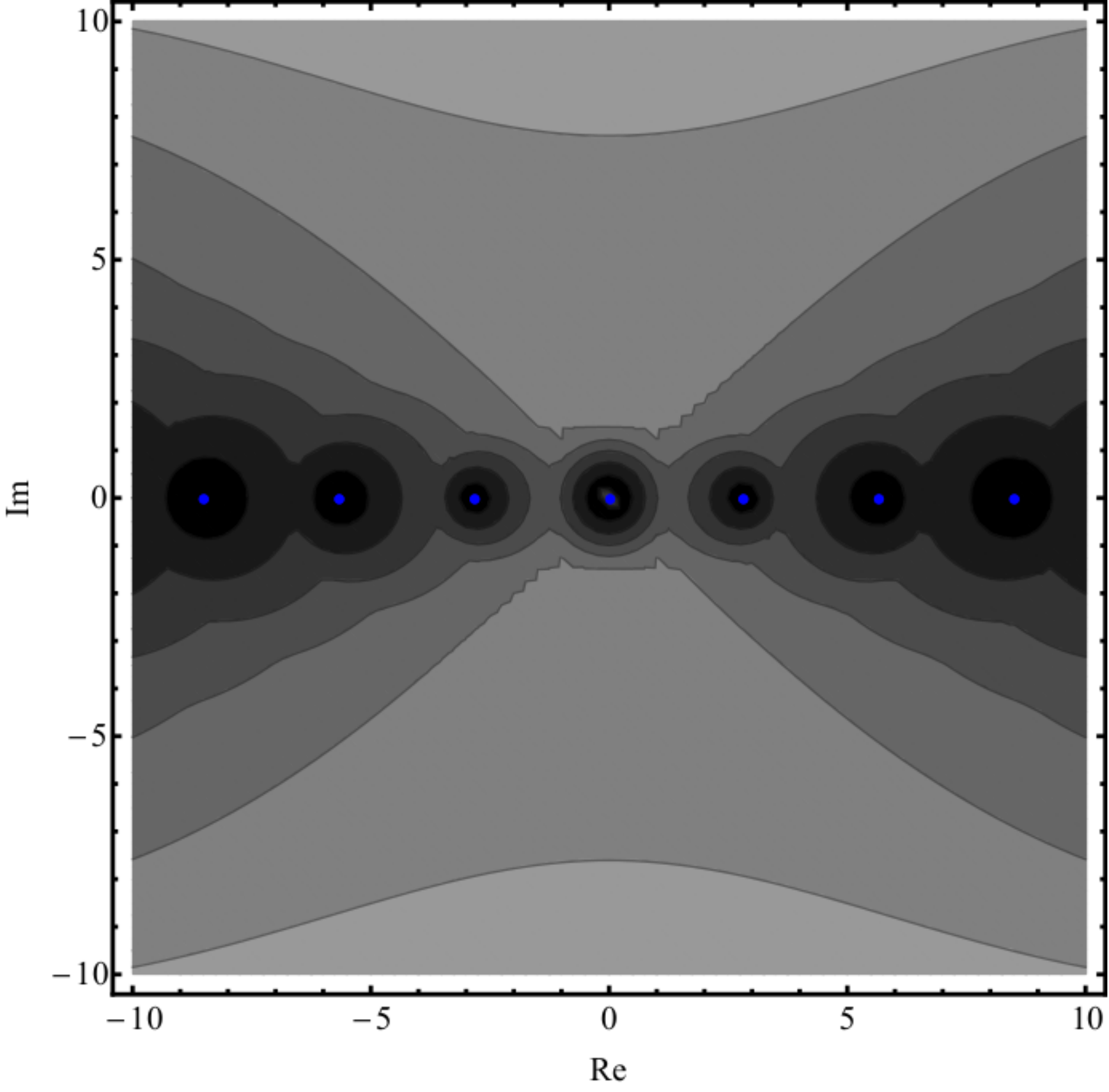}\vspace{-0.5mm}
\caption{\small Eigenvalues (blue dots) and $\eps$-pseudospectra of the truncated $4n\times 4n$ matrices $A_n$ for $n=2$ (top), $n=4$ (middle), $n=6$ (bottom) and $\eps=1.5,1.4,\dots,0.6,0.5$.
\label{figConstResNorm}}
\end{center}
\end{figure}
In Figure~\ref{figConstResNorm} the eigenvalues (blue dots) and nested $\eps$-pseudospectra (different shades of grey) of $A_n$ are shown for $n=2,4,6$ and $\eps=0.5,0.6,\dots,1.4,1.5$. 
As $n$ is increased, for $\eps> 1$ the $\eps$-pseudospectra grow and seem to fill the whole complex plane, whereas for $\eps\leq 1$ they converge to $\sigma_{\eps}(A)\neq\C$.
We prove these observations more rigorously. 
In fact, we show that there exists an open subset of the complex plane where the resolvent norms of $A$ and $T$ are constant ($1/\eps=1$). So we cannot conclude $\eps$-pseudospectral exactness using Theorem~\ref{thmblockdiag}~iii). However, $\eps$-pseudospectral inclusion follows from Theorem~\ref{thmexactness2}~i). In addition, the upper block-triangular form of $A$ implies that if $x\in H_n$, then
$P_n(A-\lm)^{-1}x=(A_n-\lm)^{-1}x_n$. This yields $\|(A_n-\lm)^{-1}\|\leq \|(A-\lm)^{-1}\|$ and so
$$\overline{\sigma_{\eps}(A_n)}\subset\overline{\sigma_{\eps}(A)}, \quad n\in\N.$$
Hence no $\eps$-pseudospectral pollution occurs.

We calculate, for $\lm=r\e^{\I\varphi}$ with $\re(\lm^2)=r^2\cos(2\varphi)<0$,
\beq\label{eq.resnorm.Tj}
\|(T_j-\lm)^{-1}\|^2=\frac{1}{|a_j-\lm^2|^2}\left\|\bmat \lm & 1 \\ a_j & \lm\emat\right\|^2
\leq \frac{(r+\max\{a_j,1\})^2}{r^4+2a_jr^2|\cos(2\varphi)|+a_j^2}.
\eeq
Hence, as in~\cite[Example~3.7]{Boegli-2014-80}, the resolvent norm of $T$ is constant on a non-empty open set,
$$\|(T-r\e^{\I\varphi})^{-1}\|=1 \quad \text{if}\quad \cos(2\varphi)<0,\quad r\geq\max\left\{\frac{1+\sqrt 5}{2},\frac{1}{|\cos(2\varphi)|}\right\}.$$
One may check that 
$$\sigma_{\rm ess}(T)=\sigma_{\rm ess}(T^*)^*=\emptyset, \quad \underset{K\text{ compact}}{\bigcap}\sigma(T+K)=\emptyset.$$
In addition, since $\|(T_j-\lm)^{-1}\|\to 1$, $j\to\infty$, for any $\lm\in\rho(T)$, we have $\Lambda_{{\rm ess},\eps}(T)=\Lambda_{{\rm ess},\eps}(T^*)^*=\emptyset$, $\eps\neq 1$.
Therefore, using Theorems~\ref{thmconstresnorm}~i) and~\ref{thmblockdiag}, 
$$\sigma_{\rm ess}(A)=\sigma_{\rm ess}(A^*)^*=\emptyset, \quad 
\Lambda_{{\rm ess},\eps}(A)=\Lambda_{{\rm ess},\eps}(A^*)^*=\begin{cases}\C, &\eps= 1, \\ \emptyset, &\eps\neq 1.\end{cases}$$
This implies, in particular, that $\|(A-\lm)^{-1}\|\geq 1$ for all $\lm\in\rho(A)$.
Now we prove that the resolvent norm is constant ($=1$) on an open set.
To this end, let $\varphi$ be so that $\cos(2\varphi)<-\frac{1}{4}$.
We show that there exists $r_{\varphi}>0$ such that
\beq\label{eq.const.A}
\|(A-r\e^{\I\varphi})^{-1}\|=1, \quad r\geq r_{\varphi}.
\eeq

Define
$$\delta_{-1}:=\frac{1}{3}, \quad \delta_j:=\frac{1}{a_{j+1}|\cos(2\varphi)|}, \quad j\in\Z\backslash\{-1\}.$$
Then $\delta_j\to 0$ as $|j|\to\infty$ and 
$$\delta_{j-1}a_j=\frac{1}{|\cos(2\varphi)|}, \quad j\in\Z\backslash\{0\}, \quad \delta:=\sup_{j\in\Z}\delta_j=\max\left\{\frac{1}{3},\frac{1}{a_1 |\cos(2\varphi)|}\right\}<\frac{1}{2}.$$
Define functions $f_j:[0,\infty)\to\R$, $j\in\Z$, by
$$f_0(r):= r^4(1-\delta)-\frac{r^2+1}{\delta_{-1}}-2r$$
and, for $j\neq 0$,
\begin{align*}
f_j(r):=& \,r^4(1-\delta)+1\\ &+a_j r^2\bigg(1-2\delta)|\cos(2\varphi)|-\frac{2}{r}
- \sup_{j\in\Z\backslash\{-1\}}\frac{a_j}{a_{j+1}}\frac{1}{|\cos(2\varphi)|}-|\cos(2\varphi)|\bigg).
\end{align*}
One may verify that there exists $r_{\varphi}>0$ such that $f_j(r)>0$ for all $j\in\Z$ and $r\geq r_{\varphi}$.
We calculate, for $\lm=r\e^{\I\varphi}$ with $r\geq r_{\varphi}$,
$$\|S_{j-1}(T_j-\lm)^{-1}\|^2=\frac{|\lm|^2+1}{|a_j-\lm^2|^2}= \frac{r^2+1}{r^4+2 a_jr^2 |\cos(2\varphi)|+a_j^2}, \quad j\in\Z.$$
We abbreviate
$$g_j(r):=r^4+2 a_jr^2 |\cos(2\varphi)|+a_j^2>0,  \quad j\in\Z.$$
Then, with~\eqref{eq.resnorm.Tj}, we estimate for $j\neq 0$,
\begin{align*}
&1-\delta_j-\left(\frac{1}{\delta_{j-1}}-1\right)\|S_{j-1}(T_{j}-\lm)^{-1}\|^2 -\|(T_j-\lm)^{-1}\|^2\\
&=\frac{r^4(1-\delta_j)+1+a_j\left(r^2\left(2(1-\delta_j)|\cos(2\varphi)|-\frac{1}{\delta_{j-1}a_j}-\frac{2}{r}\right)-\delta_ja_j-\frac{1}{\delta_{j-1}a_j}\right)}{g_j(r)}\\
&\geq \frac{f_j(r)}{g_j(r)}>0,
\end{align*}
and analogously for $j=0$.
So we arrive at
$$1-\delta_j-\left(\frac{1}{\delta_{j-1}}-1\right)\|S_{j-1}(T_{j}-\lm)^{-1}\|^2 > \|(T_j-\lm)^{-1}\|^2, \quad j\in\Z.$$
Let $x=(x_j)_{j\in\Z}\in\dom(A)=\dom(T)$ with $x_j=(u_j,v_j)^t\in\C^2$. Then
\begin{align*}
&\|(A-\lm)x\|^2\\
&=\|(T-\lm)x+Sx\|^2
=\sum_{j\in\Z}\|(T_j-\lm)x_j+S_{j}x_{j+1}\|^2\\
&\geq \sum_{j\in\Z}(1-\delta_j)\|(T_j-\lm)x_j\|^2-\left(\frac{1}{\delta_j}-1\right)\|S_jx_{j+1}\|^2\\
&\geq \sum_{j\in\Z}(1-\delta_j)\|(T_j-\lm)x_j\|^2-\left(\frac{1}{\delta_j}-1\right)\|S_j(T_{j+1}-\lm)^{-1}\|^2 \|(T_{j+1}-\lm)x_{j+1}\|^2\\
&= \sum_{j\in\Z}\left(1-\delta_j-\left(\frac{1}{\delta_{j-1}}-1\right)\|S_{j-1}(T_{j}-\lm)^{-1}\|^2 \right)\|(T_j-\lm)x_j\|^2\\
&\geq \sum_{j\in\Z}\|(T_j-\lm)^{-1}\|^2\|(T_j-\lm)x_j\|^2\geq \|x\|^2,
\end{align*}
which implies $\|(A-\lm)^{-1}\|\leq 1$ and hence~\eqref{eq.const.A}.
\end{example}

\subsection{Domain truncation of PDEs on $\R^d$}\label{subsectionPDE}
In this application we study the sum of two partial differential operators in $L^2(\R^d)$, the first one of order $k\in\N$ and the second one is of lower order and relatively compact. To this end, we use a multi-index $\alpha=(\alpha_1,\dots,\alpha_d)^t\in\R^d$ with $|\alpha|:=\alpha_1+\dots+\alpha_d$ and 
$$D^{\alpha}:=\frac{\rd^{|\alpha|}}{\rd x_1^{\alpha_1}\cdots\rd x_d^{\alpha_d}}, \quad
\zeta^{\alpha}:=\zeta_1^{\alpha_1}\cdots\zeta_d^{\alpha_d}, \quad \zeta=(\zeta_1,\dots,\zeta_d)^t\in\R^d.$$ 
The differential expressions are of the form
$$\tau:=\tau_1+\tau_2, \quad \tau_1:=\sum_{|\alpha|\leq k} \frac{1}{\I^{|\alpha|}} c_{\alpha}D^{\alpha},\quad \tau_2:=\sum_{|\alpha|\leq k-1}\frac{1}{\I^{|\alpha|}}  b_{\alpha} D^{\alpha},$$
where  $c_{\alpha}\in\C$. 
In order to reduce the technical difficulties, we assume that the functions $b_{\alpha}:\R\to \C$ are sufficiently smooth,
$$b_{\alpha}\in W^{|\alpha|,\infty}(\R^d), \quad |\alpha|\leq k-1.$$
In addition, suppose that  
\beq\label{ass.decay}
\lim_{|x|\to\infty}D^{\beta} b_{\alpha}(x)=0, \quad |\beta|\leq |\alpha|\leq k-1.
\eeq
Define the \emph{symbol} $p:\R^d\to\C$ and \emph{principal symbol} $p_k:\R^d\to\C$ by 
$$p(\zeta):=p_k(\zeta)+\sum_{|\alpha|\leq k-1}c_{\alpha}\zeta^\alpha, \quad p_k(\zeta):=\sum_{|\alpha|=k}c_{\alpha}\zeta^\alpha.$$
We assume that $p$ is elliptic, i.e.\ $$p_k(\zeta)\neq 0, \quad \zeta\in\R^d\backslash\{0\}.$$

For $n\in\N$ let $P_n$ be the orthogonal projection of $L^2(\R^d)$ onto $L^2\big((-n,n)^d\big)$, given by multiplication with the characteristic function $\chi_{(-n,n)^d}$. It is easy to see that $P_n\s I$.

\begin{theorem}\label{thmdiffop}
Let $A$ and $A_n$, $n\in\N$, be realisations of $\tau$ in $L^2(\R)$ and $L^2\big((-n,n)^d\big)$, $n\in\N$, respectively, with domains
\begin{align*}
\dom(A)&:=W^{k,2}(\R^d), \\ 
\dom(A_n)&:=\big\{f\in W^{k,2}\big((-n,n)^d\big):\,D^{\alpha}f|_{\{x_j=-n\}}=D^{\alpha}f|_{\{x_j=n\}}, \,j\leq d,\,|\alpha|\leq k-1\big\}.
\end{align*}
\begin{enumerate}
\item[\rm i)] We have  $A_n\gsr A$ and $A_n^*\gsr A^*$.
\item[\rm ii)]  The limiting essential spectra satisfy 
\begin{align}\label{eq.sigmae.diffop}
\sigma_{\rm ess}\big((A_n)_{n\in\N}\big)&=\sigma_{\rm ess}\big((A_n^*)_{n\in\N}\big)^*=\sigma_{\rm ess}(A)=\{p(\zeta):\,\zeta\in\R^d\};
\end{align}
hence no spectral pollution occurs for the approximation $(A_n)_{n\in\N}$ of~$A$, and every isolated $\lm\in\sigma_{\rm dis}(A)$ is the limit of a sequence $(\lm_n)_{n\in\N}$ with $\lm_n\in\sigma(A_n)$, $n\in\N$.
\item[\rm iii)] The limiting essential $\eps$-near spectra satisfy 
\beq \label{eq.Lambdae.diffop}
\begin{aligned}
\Lambda_{{\rm ess},\eps}\big((A_n)_{n\in\N}\big)&=\Lambda_{{\rm ess},\eps}\big((A_n^*)_{n\in\N}\big)^*=\Lambda_{{\rm ess},\eps}(A)\\
&=\{p(\zeta)+z:\,\zeta\in\R^d,\,|z|= \eps\}\subset\overline{\sigma_{\eps}(A)},
\end{aligned}
\eeq
and so $(A_n)_{n\in\N}$ is an $\eps$-pseudospectrally exact approximation of $A$.
\end{enumerate}
\end{theorem}

\begin{proof}
Let $T,S$ and $T_n,S_n$, $n\in\N$, be the realisations of $\tau_1, \tau_2$ in $L^2(\R)$ and $L^2\big((-n,n)^d\big)$, $n\in\N$, respectively, with domains
\begin{align*}
\dom(T)&=\dom(S):=\dom(A), \quad
\dom(T_n)=\dom(S_n):=\dom(A_n), \quad n\in\N.
\end{align*}
The operators $T$ and $T_n$, $n\in\N$, are normal; the symbol of the adjoint operators $T^*$, $T_n^*$, $n\in\N$, is simply the complex conjugate symbol $\overline{p}$.
For $f\in L^2(\R^d)$ denote its Fourier transform by $\widehat f$ and
for $n\in\N$ and $f_n\in L^2((-n,n)^d)$ denote by $\widehat f_n=(\widehat f_n(\zeta))_{\zeta\in\Z^d}\in l^2(\Z^d)$ the complex Fourier coefficients, i.e.\
\begin{align*}
\widehat f(\zeta)&:=\frac{1}{(2\pi)^{\frac{d}{2}}}\int_{\R^d}f(x)\e^{-\I\zeta \cdot x}\rd x, \quad \zeta\in\R^d,\\
\widehat f_n(\zeta)&:=\frac{1}{(2n)^{\frac{d}{2}}}\int_{(-n,n)^d}f(x)\e^{-\I\frac{\pi}{n}\zeta \cdot x}\rd x, \quad \zeta\in\Z^d.\end{align*}
Parseval's identity yields that, if $f\in\dom(T)$, $f_n\in\dom(T_n)$,
\begin{align*}
\|(T-\lm)f\|&=\|(p-\lm)\widehat f\|, \quad \|(T_n-\lm)f_n\|=\Big\|\left(p\left(\cdot\, \frac{\pi}{n}\right)-\lm\right)\widehat f_n\Big\|_{l^2(\Z^d)};
\end{align*}
moreover, if $\lm\notin\{p(\zeta):\,\zeta\in\R^d\}$, then
$$\|(T-\lm)^{-1}f\|=\|(p-\lm)^{-1}\widehat f\|, \quad \|(T_n-\lm)^{-1}f_n\|=\Big\|\left(p\left(\cdot\, \frac{\pi}{n}\right)-\lm\right)^{-1}\widehat f_n\Big\|_{l^2(\Z^d)}.$$
We readily conclude
\begin{align*}
\sigma_{\rm ess}(T)&=\sigma(T)=\{p(\zeta):\,\zeta\in\R^d\}, \\
\sigma_{\eps}(T)&=\{\lm\in\C:\,{\rm dist}(\lm,\sigma(T))<\eps\}, 
\quad \Lambda_{{\rm ess},\eps}(T)=\{p(\zeta)+z:\,\zeta\in\R^d,\,|z|=\eps\},\\
\sigma(T_n)&=\Big\{p\Big(\zeta\frac{\pi}{n}\Big):\,\zeta\in\Z^d\Big\},\quad \sigma_{\eps}(T_n)=\{\lm\in\C:\,{\rm dist}(\lm,\sigma(T_n))<\eps\}, \quad n\in\N,
\end{align*}
and 
\begin{align*}
\sigma_{\rm ess}\big((T_n)_{n\in\N}\big)&\subset \{p(\zeta):\,\zeta\in\R^d\}=\sigma_{\rm ess}(T), \\
\Lambda_{{\rm ess},\eps}\big((T_n)_{n\in\N}\big)&\subset \{p(\zeta)+z:\,\zeta\in\R^d,\,|z|=\eps\}=\Lambda_{{\rm ess},\eps}(T),
\end{align*}
and the latter are equalities by Propositions~\ref{propsigmaess} and~\ref{proplimsigmaess}.
The same identities hold for the adjoint operators. This proves~\eqref{eq.sigmae.diffop} and~\eqref{eq.Lambdae.diffop}.

For any $\Omega\subset\R^d$ we have
\beq\label{eq.estimateProd.Omega}
\begin{aligned}
\|\chi_{\Omega}S(T-\lm)^{-1}f\|&\leq \sum_{|\alpha|\leq k-1}\|b_{\alpha}\|_{L^{\infty}(\Omega)}\sup_{\zeta\in\R^d}\Big|\zeta^{\alpha}\Big(p\Big(\zeta\frac{\pi}{n}\Big)-\lm\Big)^{-1}\Big|\| f\|,\\
\|\chi_{\Omega}S_n(T_n-\lm)^{-1}f_n\|&\leq \sum_{|\alpha|\leq k-1}\|b_{\alpha}\|_{L^{\infty}(\Omega)}\sup_{\zeta\in\R^d}\Big|\zeta^{\alpha}\Big(p\Big(\zeta\frac{\pi}{n}\Big)-\lm\Big)^{-1}\Big| \| f_n\|.
\end{aligned}
\eeq
By setting $\Omega=\R^d$, we see that $\|S(T-\lm)^{-1}\|$, $\|S_n(T_n-\lm)^{-1}\|$, $n\in\N$, are uniformly bounded, and the uniform bound can be arbitrarily small by choosing $\lm$ far away from $\{p(\zeta):\,\zeta\in\R^d\}$.
The same argument also holds for the adjoint operators.
Let $\lm_0$ be so that
\beq\label{eq.prodestimate1}
\begin{aligned}
\|S(T-\lm_0)^{-1}\|<1, \quad \sup_{n\in\N}\|S_n(T_n-\lm_0)^{-1}\|<1, \\
\|S^*(T^*-\overline{\lm_0})^{-1}\|<1, \quad \sup_{n\in\N}\|S_n^*(T_n^*-\overline{\lm_0})^{-1}\|<1.
\end{aligned}
\eeq
Hence, in particular, $S$, $S_n$, $S^*$, $S_n^*$ are respectively $T$-, $T_n$-, $T^*$-, $T_n^*$-bounded with relative bounds $<1$ and so, by~\cite[Corollary~1]{Hess-Kato}, $A^*=T^*+S^*$ and $A_n^*=T_n^*+S_n^*$. 

By the assumptions~\eqref{ass.decay} and~\cite[Theorem~IX.8.2]{edmundsevans}, the operator $S$ is $T$-compact and $S^*$ is $T^*$-compact. Hence~\cite[Theorem~IX.2.1]{edmundsevans} implies
$$\sigma_{{\rm ess}}(A)=\sigma_{\rm ess}(T), \quad \sigma_{\rm ess}(A^*)^*=\sigma_{\rm ess}(T^*)^*.$$

Now we show that the assumptions (a)--(c) of Theorem~\ref{prop.pert.pseudo}~ii),~iii) are satisfied for both $A=T+S$, $A_n=T_n+S_n$ and $A^*=T^*+S^*$, $A_n^*=T_n^*+S_n^*$; then the claims i)--iii) follow from the above arguments and together with Theorems~\ref{mainthmspectralexactness},~\ref{thmexactness2}.
We prove (c) before (b) as the proof of the latter relies on the former.

(a)
The estimates are satisfied by the choice of $\lm_0$, see~\eqref{eq.prodestimate1}.

(c)
Let $f\in C_0^{\infty}(\R^d)$, and let $n_f\in\N$ be so large that ${\rm supp}f\subset (-n_f,n_f)^d$. Then, for $n\geq n_f$,
\begin{align*}
(T_n-\lm_0)^{-1}P_n f&=P_n(T-\lm_0)^{-1}f, \\
 (T_n^*-\overline{\lm_0})^{-1}P_nf&=P_n(T^*-\overline{\lm_0})^{-1}f,\\
(S_n(T_n-\lm_0)^{-1})^*P_n f&=(T_n^*-\overline{\lm_0})^{-1}S_n^*P_nf=P_n(T^*-\overline{\lm_0})^{-1}S^*f\\
&=P_n(S(T-\lm_0)^{-1})^*f,\\
(S_n^*(T_n^*-\overline{\lm_0})^{-1})^*P_n f&=(T_n-\lm_0)^{-1}S_nP_nf=P_n(T-\lm_0)^{-1}Sf\\
&=P_n(S^*(T^*-\overline{\lm_0})^{-1})^*f.
\end{align*}
Now the claimed strong convergences follow using~\eqref{eq.prodestimate1} and the density of $C_0^{\infty}(\R^d)$ in $L^2(\R^d)$.

(b)
We prove that 
$\big(S_n(T_n-\lm_0)^{-1}\big)_{n\in\N}$ is a discretely compact sequence; for $\big(S_n^*(T_n^*-\overline{\lm_0}\big)^{-1})_{n\in\N}$ the argument is analogous.
To this end, let $I\subset\N$ be an infinite subset and let $f_n\in L^2((-n,n)^d)$, $n\in I$, be a bounded sequence. Then there exists an infinite subset $I_1\subset I$ such that $(f_n)_{n\in I_1}$ is  weakly convergent in $L^2(\R^d)$; denote the weak limit by $f$. We show that $\|S_n(T_n-\lm_0)^{-1}f_n-S(T-\lm_0)^{-1}f\|\to 0$ as $n\in I_1$, $n\to\infty$.
Assume that the claim is false, i.e.\ there exist an infinite subset $I_2\subset I_1$  and $\delta>0$ so that
\beq\label{eq.disccomp.contra}
\|S_n(T_n-\lm_0)^{-1}f_n-S(T-\lm_0)^{-1}f\|^2\geq\delta, \quad n\in I_2.
\eeq
Note that (c) and Lemma~\ref{lemmadisccompletelycont}~i) imply that $(S_n(T_n-\lm)^{-1}f_n)_{n\in I_2}$ converges weakly to $S(T-\lm_0)^{-1}f$.
The assumption~\eqref{ass.decay} yields 
$$\lim_{n\to\infty}\|b_{\alpha}\|_{L^{\infty}(\R^d\backslash (-n,n)^d)}= 0, \quad |\alpha|\leq k-1.$$ 
Hence, by~\eqref{eq.estimateProd.Omega},
there exists $n_0\in\N$ so large that 
$$\|\chi_{\R^d\backslash (-n_0,n_0)^d}S_n(T_n-\lm)^{-1}f_n-\chi_{\R^d\backslash (-n_0,n_0)^d}S(T-\lm)^{-1}f\|^2< \frac{\delta}{2}$$
for all $n\in I_2$ with $n\geq n_0$; denote by $I_3$ the set of all such $n$.
An estimate similar to~\eqref{eq.estimateProd.Omega} reveals that the $W^{k,2}((-n_0,n_0)^d)$ norms $$\|\chi_{(-n_0,n_0)^d}(T_n-\lm_0)^{-1}f_n-\chi_{(-n_0,n_0)^d}(T-\lm_0)^{-1}f\|_{W^{k,2}((-n_0,n_0)^d)}, \quad n\in I_3,$$ 
are uniformly bounded, and so the Sobolev embedding theorem yields  
$$\|\chi_{(-n_0,n_0)^d}S_n(T_n-\lm_0)^{-1}f_n-\chi_{(-n_0,n_0)^d}S(T-\lm_0)^{-1}f\|_{L^{2}((-n_0,n_0)^d)}\tolong 0.$$ 
Altogether we arrive at a contradiction to~\eqref{eq.disccomp.contra}, which proves the claim.
\end{proof}

It is convenient to represent $A_n$ with respect to the Fourier basis and, in a further approximation step, to truncate the infinite matrix to finite sections. We prove that these two approximation processes can be performed in one.
For $n\in\N$ let $$e_{\zeta}^{(n)}(x):=\frac{1}{(2n)^{\frac{d}{2}}}\,\e^{\I \frac{\pi}{n} \zeta\cdot x}, \quad x\in (-n,n)^d,\quad \zeta \in\Z^d,$$ 
denote the Fourier basis of $L^2((-n,n)^d)$. Note that the basis functions belong to $\dom(A_n)$. 
Let $Q_n$ denote the orthogonal projection of $L^2((-n,n)^d)$ onto 
\beq\label{eq.nfourier}
{\rm span}\big\{e_{\zeta}^{(n)}:\,\zeta\in\Z^d,\,\|\zeta\|_{\infty}\leq n\big\}.
\eeq 
One may check that $Q_n\s I$ and hence $Q_nP_n\s I$.

\begin{theorem}
The claims {\rm i)--iii)} of Theorem~{\rm\ref{thmdiffop}} continue to hold if $A_n$ is replaced by $A_{n;n}:=Q_nA_n|_{\ran(Q_n)}$.
\end{theorem}

\begin{proof}
Define $T_{n;n}:=Q_nT_n|_{\ran(Q_n)}$, $n\in\N$.
 Note that $Q_nT_n=T_{n;n}Q_n$, $n\in\N$.
Analogously as in the proof of Theorem~{\rm\ref{thmdiffop}}, we obtain
$$\sigma_{\rm ess}\big((T_{n;n})_{n\in\N}\big)= \sigma_{\rm ess}(T), \quad \Lambda_{{\rm ess},\eps}\big((T_{n;n})_{n\in\N}\big)= \Lambda_{{\rm ess},\eps}(T),$$
and the respective equalities for the adjoint operators.
It is easy to see that $S_{n;n}:=Q_nS_n|_{\ran(Q_n)}$, $n\in\N$, satisfy
$$S_{n;n}(T_{n;n}-\lm_0)^{-1}=Q_nS_n(T_n-\lm)^{-1}|_{\ran(Q_n)}, \quad n\in\N,$$
and hence the discrete compactness of $\big(S_{n;n}(T_{n;n}-\lm_0)^{-1}\big)_{n\in\N}$ follows from the one of $\big(S_n(T_n-\lm_0)^{-1}\big)_{n\in\N}$ and from $Q_n\s I$. By an analogous reasoning, the sequence $\big(S_{n;n}^*(T_{n;n}^*-\overline{\lm_0}\big)^{-1})_{n\in\N}$ is discretely compact.
The rest of the proof follows the one of  Theorem~{\rm\ref{thmdiffop}}.
\end{proof}

\begin{example}
Let $d=1$ and consider the constant-coefficient differential operator 
$$T:=-\frac{\rd^2}{\rd x^2}-2\,\frac{\rd}{\rd x}, \quad \dom(T):=W^{2,2}(\R).$$
The above assumptions are satisfied if we perturb $T$ by $S=b$ with a potential $b\in L^{\infty}(\R)$ such that $|b(x)|\to 0$ as $|x|\to\infty$.
For  $$b(x):=20\sin(x)\e^{-x^2}, \quad x\in\R,$$ 
the numerically found eigenvalues and pseudospectra of the operator $T+S$ truncated to the $(2n-1)$-dimensional subspace in~\eqref{eq.nfourier} are shown in Figure~\ref{fig.CC}.
\begin{figure}[ht]
\begin{center}
 \includegraphics[width=0.47\textwidth]{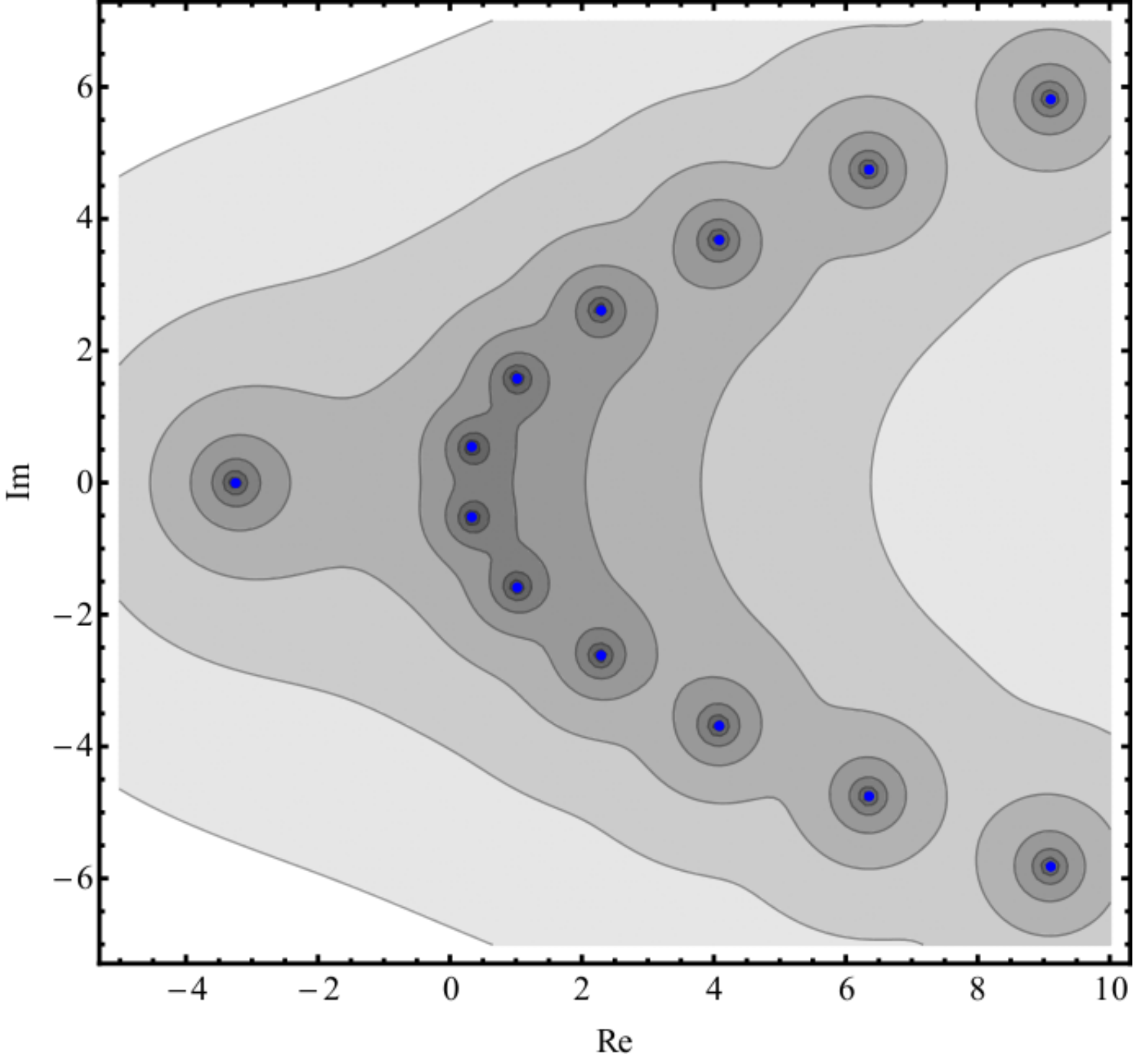}\hfill
 \includegraphics[width=0.47\textwidth]{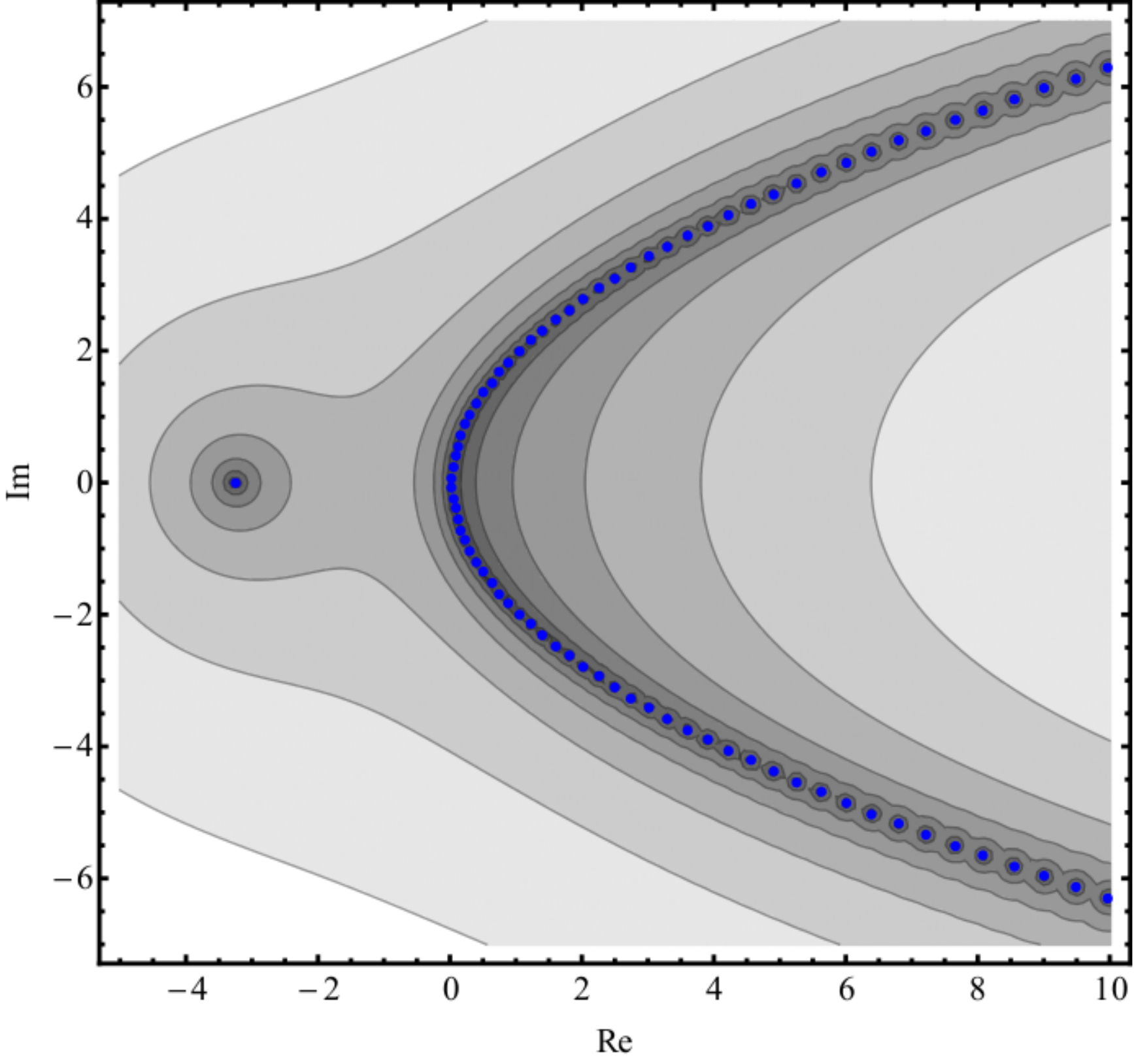}
\caption{\small Eigenvalues (blue dots) and $\eps$-pseudospectra for $\eps=2^3,2^2,\dots,2^{-3}$ in interval $[-5,10]+[-7,7]\,\I$ of approximation $A_{n;n}$ for $n=10$ (left) and $n=100$ (right).}
\label{fig.CC}
\end{center}
\end{figure}
The approximation is spectrally and $\eps$-pseudospectrally exact; the only discrete eigenvalue in the box $[-5,10]+[-7,7]\,\I $ is $\lm\approx -3.25$.
\end{example}

\subsection*{Acknowledgements}
The first part of the paper is based on results of the author's Ph.D.\ thesis~\cite{boegli-phd}. She would like to thank her doctoral advisor Christiane Tretter for the guidance.
The work was supported by the Swiss National Science Foundation (SNF), grant no.\ 200020\_146477 and Early Postdoc.Mobility project P2BEP2\_159007.

{\footnotesize
\bibliographystyle{acm}
\bibliography{mybib}
}

\end{document}